\documentclass[11pt,reqno]{amsart}
\usepackage[T1]{fontenc}
\usepackage{graphicx,   amssymb, xcolor}
\usepackage{cite}
\usepackage{fixltx2e}
\usepackage{hyperref}

\usepackage{color}
\definecolor{MyLinkColor}{rgb}{0,0,0.4}

\newcommand{\vdiv}{\mathop{\rm div}}
\newcommand{\R}{{\mathbb R}}
\newcommand{\RRM}{{\mathbb R}}

\newcommand{\bA}{{\mathbb A}}
\newcommand{\bT}{{\mathbb T}}
\newcommand{\bB}{\mathbb{B}}
\newcommand{\bD}{\mathbb{D}}

\newcommand{\N}{{\mathbb N}}

\newcommand{\bV}{{\mathbb V}}

\newcommand{\kL}{\mathcal{L}}

\newcommand{\cP}{\mathcal{P}}
\newcommand{\cU}{\mathcal{U}}
\newcommand{\clu}{\mathcal{U}}
\newcommand{\clp}{\mathcal{P}}

\newcommand{\PV}{\mathop{\rm PV}\nolimits}
\newcommand{\ov}{\overline}
\newcommand{\p}{\partial}

\newcommand{\e}{\varepsilon}

\newcommand{\0}{\Omega}
\newcommand{\G}{\Gamma}

\newcommand{\be}{\begin{equation}}
\newcommand{\ee}{\end{equation}}
\newcommand{\wtdu}{\widetilde{(\nabla u)^\pm}}
\newcommand{\wtp}{\widetilde{\Pi^\pm}}
\newcommand{\wtt}{\widetilde{T_1^\pm}}
\newcommand{\wtdtwou}{\widetilde{\p_2u^\pm}}

\newtheorem{thm}{Theorem}[section]
\newtheorem{prop}[thm]{Proposition}
\newtheorem{lemma}[thm]{Lemma}
\newtheorem{cor}[thm]{Corollary}
\theoremstyle{remark}

\numberwithin{equation}{section}
\begin{document}

\title[One-phase Stokes flow by capillarity]{Capillarity driven Stokes flow:\\ the one-phase problem as small viscosity limit}
\thanks{Partially supported by DFG Research Training Group~2339 ``Interfaces, Complex Structures, and Singular Limits in Continuum Mechanics - Analysis and Numerics''}

\author{Bogdan--Vasile Matioc}
\address{Fakult\"at f\"ur Mathematik, Universit\"at Regensburg,   93040 Regensburg, Deutschland.}
\email{bogdan.matioc@ur.de}
  
\author{Georg Prokert}
\address{Faculty of Mathematics and Computer Science, Technical University Eindhoven,   The Netherlands.}
\email{g.prokert@tue.nl}


\begin{abstract}
 We consider the quasistationary Stokes flow that describes the motion of a two-dimensional fluid body 
under the influence of surface tension effects in an unbounded, infinite-bottom geometry. 
We reformulate the problem as a fully nonlinear parabolic evolution problem for the function that parameterizes  the boundary of the fluid with the nonlinearities
 expressed in terms of singular integrals.
We  prove well-posedness of the problem in the  subcritical Sobolev spaces $H^s(\mathbb{R})$  up to critical regularity,  
 and establish   parabolic smoothing properties for the solutions.
Moreover, we identify the problem as the singular limit of the two-phase quasistationary   Stokes flow   when the viscosity of one of the fluids vanishes.
\end{abstract}

\maketitle

\section{Introduction}

In this paper we consider the two-dimensional flow of a  fluid layer $\0(t) $ of infinite depth
 in the case when the motion of  the incompressible fluid is governed  by the quasistationary Stokes equations 
  and the motion is driven by surface tension at the free boundary $\Gamma(t)=\p\0(t)$. 
  We consider the one-phase problem, i.e. no forces are exerted on the liquid by the medium above it.
 The mathematical model is  given  by the following system of equations
\begin{subequations}\label{STOKES}
\begin{equation}\label{StP}
\left.
\begin{array}{rclll}
\mu\Delta v-\nabla p&=&0&\mbox{in $\Omega(t)$,}\\
\vdiv v&=&0&\mbox{in $\Omega(t)$,}\\{}
T_\mu(v, p)\tilde\nu&=&\sigma\tilde\kappa\tilde\nu&\mbox{on $\Gamma(t)$,}\\
(v, p)(x)&\to&0&\mbox{for $|x|\to\infty$,}\\
V_n&=& v\cdot\tilde \nu&\mbox{on $\Gamma(t)$}
\end{array}\right\}
\end{equation}
for $t>0$, where the  interface $\Gamma(t)$ at time $t$ is given as a graph of a function $f(t,\cdot):\RRM\longrightarrow\RRM$, i.e the fluid domain $\0(t)$ and its boundary $\Gamma(t)$ are defined by
\begin{align*}
\0(t)&:=\{x=(x_1,x_2)\in\R^2\,:\, x_2<f(t,x_1)\},\\[1ex]
\G(t)&:=\p\0(t):=\{(\xi, f(t,\xi))\,:\, \xi\in\R\}.
\end{align*}
Additionally, the  interface $\Gamma(t)$ is assumed to be known at time~${t=0}$,  i.e.
 \begin{align}\label{IC}
 f(0,\cdot)=f^{(0)}.
 \end{align}
 \end{subequations}
 In Eq. \eqref{StP} above, $v=v(t):\0(t)\longrightarrow\R^2 $ and $p=p(t):\0(t)\longrightarrow\R$ 
 are  the velocity and the pressure of the Newtonian fluid, $\tilde\nu=(\tilde\nu^1, \tilde\nu^2)$ is the unit exterior normal to~$\p\Omega$, 
 $\tilde\kappa$ denotes the curvature of the interface  (negative where $\Omega(t)$ is convex), and  
 $T_\mu(v,p)=(T_{\mu,ij}(v,p))_{1\leq i,\, j\leq 2}$ is the stress tensor  which is given by
\begin{align}\label{defT}
T_{\mu}(v, p):=- pE_2+\mu\big[\nabla v+(\nabla v)^\top\big], \quad (\nabla v)_{ij}:=\p_jv_i.
\end{align}
Moreover, $V_n$ is the normal velocity of the interface~$\Gamma(t)$,  $a\cdot b$ denotes the Euclidean scalar product of two vectors $a,\, b\in\R^2$,  $E_2\in \R^{2\times 2}$ is the identity matrix,
 and the positive constants~$\mu$ and~$\sigma$  are the 
dynamic viscosity of the fluid  and the surface tension coefficient at the interface~$\Gamma(t),$ respectively. 

Previous analysis related to \eqref{StP} considered mainly the case of a sufficiently regular bounded fluid domain~$\0(t)$. 
More precisely, in \cite{PG97} the authors studied the quasistationary motion of a free capillary liquid drop in  $\R^d$ for initial data in    $H^{s+1}(\Sigma)$, $s\geq s_1,$ 
 $s_1$ being the smallest  integer that satisfies   $s_1>3+(d-1)/2$  and   $\Sigma\subset\R^d$ the smooth boundary of a strictly star shaped domain in $\R^d$.
 The authors established in \cite{PG97} the well-posedness of the problem and they also showed  that the equilibria of the problem, which are  balls, are exponentially stable.
 In the context when $\Sigma $ is the boundary of the unit ball, it is  proven in  \cite{Fr02a} (for $d=2$)  that solutions corresponding to small data in $H^5(\Sigma)$ exist globally  and converge to a ball,
  while in   \cite{FR02b} (for $d=3$)  the authors  established the same result for  small data in $H^6(\Sigma).$
Finally, in three space dimensions and for an initially bounded geometry possessing a~${\rm C}^{3+\alpha}$-boundary, $\alpha>0$,  it was shown in \cite{S99} that the quasistationary Stokes flow  is well-posed, 
and  in \cite{Solo99} the same author has rigorously  justified this problem  as the singular limit of  Navier-Stokes flow when the Reynolds number vanishes.
The local well-posedness and the stability issue for the two-phase quasistationary Stokes flow  (with or without phase transitions) in a bounded geometry in $\R^d$, with $d\geq2$, has been recently studied in \cite{PS16} 
 in the phase space~$W^{2+\mu-2/p}_p(\Sigma)$, with~$1\geq \mu>(d+2)/p$,  by using a maximal $L_p$-regularity approach.
In the context of \cite{PS16},~$\Sigma$ is  a real analytic hypersurface over which the boundary between the two fluid phases is parameterized.

 We emphasize that in the references discussed above the moving interface is at least of class ${\rm C}^2$,
 whereas the critical $L_2$-Sobolev space for \eqref{STOKES} is~$H^{3/2}(\R)$, see \cite{MP2021, MP2022}.  
Our goal is to establish the well-posedness of \eqref{STOKES} in the subcritical spaces $H^{s}(\R)$ with~$s\in(3/2,2)$, see Theorem~\ref{MT1}. 
One of the obvious difficulties lies in the fact that, for $f\in H^{s}(\R)$, the curvature term in $\eqref{StP}_3$ is merely a distribution.
To handle this issue we use a strategy  inspired by the approach in the papers \cite{BaDu98, MP2021, MP2022}
where,  for the corresponding two-phase problem, potential theory was used to determine
 the velocity and pressure fields in terms  of $f$.
Such a strategy was applied also in the context of the Muskat problem, see  the surveys \cite{G17, GL20}, and it provides  
 quite optimal results as the mathematical reformulations of the problems obtained by using this strategy require less 
smallness and regularity assumptions on the data compared to other approaches based on Lagrangian or Hanzawa transformations.
  
The first goal of this paper is to show that, at each time instant $t>0$,  the free boundary,  given via~${f=f(t)}$, identifies the velocity field~$v=v(t)$ and the pressure~$p=p(t)$ uniquely.
More precisely, as shown in Theorem \ref{T:1}, if $f\in H^3(\R)$, then $(v,p)$ is given by the hydrodynamic single layer potential with a density $\beta=(\beta_1,\beta_2)^\top$ which satisfies
  \begin{equation}\label{iINT}
\Big(\frac{1}{2}- \bD(f)^*\Big)[\beta']=\sigma \Big(-\Big(\frac{f'^2}{\omega+\omega^2}\Big)',\Big(\frac{f'}{\omega}\Big)'\Big)^\top=\sigma g',
\end{equation}
where $(\cdot)'$ is the derivative with respect to the spatial coordinate $\xi\in\R$  and $g=g(f)$ is defined in \eqref{defgn} below.
 Furthermore, $\bD(f)^*$ is the $L_2$-adjoint of the double layer potential~$\bD(f)$, see~\eqref{defD}, and $\omega:=(1+f'^2)^{1/2}$, see \eqref{nutau}.
Concerning \eqref{iINT},  the following issues need to be clarified: 
\begin{itemize}
\item[(i)] The invertibility of the operators $\pm 1/2- \bD(f)$ (and $\pm 1/2- \bD(f)^*$) in $\kL(H^1(\R)^2)$;
 \item[(ii)] the question whether $(\frac{1}{2}- \bD(f)^*)^{-1}g'$ is the derivative of some $\beta\in (H^2(\R))^2$.
\end{itemize}
 We remark that  these issues are new compared to the treatment of the two-phase problem.

 With respect to (i), the main step is performed in Theorem~\ref{T:L2spec} where the invertibility in~${\kL(L_2(\R)^2)}$ is established for each~$f\in{\rm BUC}^1(\R)$.
 At this point, we rely on the Rellich identities~\eqref{RELLICH1}-\eqref{RELLICH5} for the Stokes boundary value problem which have been 
  exploited, in a bounded geometry in $\R^n$ with $n\geq3$, also in~\cite{FKV88}.
 In the unbounded two-dimensional setting considered in  the present paper we provide new arguments which use, among others, 
 also a  Rellich identity  obtained  in \cite{MBV18} in the context of the Muskat  problem.
 Based on Theorem~\ref{T:L2spec}, we then show that these operators are invertible in $\kL(H^k(\R)^2)$, ${k=1,\, 2}$,  provided that~${f\in H^{k+1}(\R)}$, see Lemma~\ref{L:3}, and in 
  $\kL(H^{s-1}(\R)^2)$ when $f\in H^{s}(\R)$, see Lemma~\ref{L:4}.
  
  Concerning (ii), we prove  in Lemma~\ref{L:A1} that, given $f\in H^s(\R)$ and $\beta\in H^1(\R)^2$, the function~${\bD(f)[\beta]}$ belongs to~${H^1(\R)^2}$ and
 \begin{equation}\label{icomder}
 (\bD(f)[\beta])'=-\bD(f)^*[\beta'].
 \end{equation}
  This relation and the observation that the right side of \eqref{iINT} is a derivative enables us to essentially replace, for $f\in H^3(\R)$, Eq. \eqref{iINT} by
 \begin{equation}\label{iINT2}
\Big(\frac{1}{2}+ \bD(f)\Big)[\beta]=\sigma g,
\end{equation}
see Corollary~\ref{C:1}. 
 These properties, in particular Lemma~\ref{L:3} 
and the equivalence of \eqref{iINT} and \eqref{iINT2}, are then used to reformulate  the one-phase Stokes flow \eqref{STOKES} as the evolution problem~\eqref{NNEP1}, which has only $f$ as  unknown.
 Its well-posedness properties are summarized in Theorem~\ref{MT1} below.

Our second main result concerns the limit behavior  for $\mu^+\to 0$ of the two-phase quasistationary Stokes problem
\begin{subequations}\label{2STOKES}
\begin{equation}\label{probint2}
\left.
\begin{array}{rclll}
\mu^\pm\Delta w^\pm-\nabla q^\pm&=&0&\mbox{in $\Omega^\pm(t)$,}\\
\vdiv w^\pm&=&0&\mbox{in $\Omega^\pm(t)$,}\\{}
[w]&=&0&\mbox{on $\Gamma(t)$,}\\{}
[T_\mu(w, q)]\tilde\nu&=&-\sigma\tilde\kappa\tilde\nu&\mbox{on $\Gamma(t)$,}\\
(w^\pm, q^\pm)(x)&\to&0&\mbox{for $|x|\to\infty$,}\\
V_n&=& w^\pm\cdot\tilde \nu&\mbox{on $\Gamma(t)$}
\end{array}\right\}
\end{equation}
for $t>0$ and
 \begin{align}\label{IC2}
 f(0)= f^{(0)},
 \end{align}
 \end{subequations}
 with $\mu^-=\mu$ fixed.
In \eqref{probint2} it is again assumed that $\Gamma(t)$ is the graph of a function~$f(t)$, 
\[
\Omega^\pm(t):=\{x=(x_1,x_2)\in\R^2\,:\, x_2\gtrless f(t,x_1)\},
 \]
and $\tilde\nu$ is the unit exterior normal to~$\p\Omega^-(t)$. 
Moreover,  $w^\pm(t)$ and $ q^\pm(t)$ represent the velocity and pressure fields in $\Omega^\pm(t)$,  respectively, and~$[v]$ (respectively $[T_\mu(v, p)]$) 
is the jump of the velocity (respectively stress tensor) across the moving interface, see \eqref{defjump}  below.
We emphasize that the limit $\mu^+\to 0$ in the formulation \eqref{2STOKES} is singular because ellipticity of the underlying boundary value problem is lost in this limit.

 In \cite{MP2022},  we reformulated the two-phase Stokes problem~\eqref{2STOKES} as a nonlinear evolution equation for~$f$, see \eqref{NNEP2} below. 
 In Sections \ref{Sec:4.3} and \ref{Sec:4.4} of the present paper we prove that the right side of \eqref{NNEP2} has a limit for $\mu^+\to 0$, and the  limit is  the right side of \eqref{NNEP1}. 
In this sense, we show that the moving boundary problem~\eqref{STOKES} represents the ``regular limit'' of~\eqref{2STOKES} for $\mu^-=\mu$ and $\mu^+\to 0$. 
This property is used in Section~\ref{Sec:4.4} to  introduce the common formulation \eqref{NNEP}  that contains both evolution problems. 
 It reads
\begin{equation}\label{NNEP*}
 \frac{df}{dt}=\Phi(\mu^+,f),\quad t\geq 0,\qquad f(0)=f^{(0)},
 \end{equation}
where $\mu^+\geq0$ is viewed as a parameter.
 We point out that though this common  formulation has been derived from the Stokes flow equations under the assumption that~${f(t)\in H^3(\R)}$,
 the nonlinear and nonlocal operator $\Phi$ is well-defined when assuming only $f\in H^s(\R)$, ${s\in(3/2,2)}$,
  and this  allows us to consider \eqref{NNEP*} under these lower smoothness assumptions.
The regularity of the limit is now seen in the fact that $\Phi$ is smooth on  $[0,\infty)\times H^s(\R)$. 
For any fixed~${\mu^+>0}$ we investigated  the  problem \eqref{NNEP*}  in \cite{MP2022}. 
In particular, we showed in \cite[Theorem~1.1]{MP2022} that, given $f^{(0)}\in H^{s}(\R)$, there exists a 
unique maximal solution~$(f_{\mu^+},w_{\mu^+}^\pm,q_{\mu^+}^\pm)$   to~\eqref{2STOKES} such that
\begin{itemize}
\item[$\bullet$] $f_{\mu^+}=f_{\mu^+}(\cdot,  f^{(0)})\in {\rm C}([0,T_{+,\mu^+}),  H^{s}(\mathbb{R}))\cap {\rm C}^1([0,T_{+,\mu^+}), H^{s-1}(\mathbb{R})),$\\[-2ex]
\item[$\bullet$] $w_{\mu^+}^\pm(t)\in {\rm C}^2(\Omega^\pm(t))\cap {\rm C}^1(\overline{\Omega^\pm(t)})$, $q_{\mu^+}^\pm(t)\in {\rm C}^1(\Omega^\pm(t))\cap {\rm C}(\overline{\Omega^\pm(t)})$ 
for all ${t\in(0,T_{+,\mu^+})}$,\\[-2ex]
\item[$\bullet$] $ w_{\mu^+}^\pm(t)|_{\G(t)}\circ\Xi_{f(t)}\in H^2(\R)^2$ for all $t\in(0,T_{+,\mu^+})$,\\[-2ex]
\end{itemize}
where $T_{+,\mu^+}=T_{+,\mu^+}( f^{(0)})\in (0,\infty]$ is the maximal existence time  and $\Xi_{f(t)}(\xi):=(\xi, f(t,\xi))$ for~$\xi\in\R.$ 

 Our first main result is based on the fact that the properties of $\Phi(\mu^+,\cdot)$ and 
of its Fr\'echet derivative~$\p_f\Phi(\mu^+,\cdot)$ which were used to prove \cite[Theorem 1.1]{MP2022} remain valid also when~${\mu^+=0}$.

 \begin{thm}\label{MT1} Let  $s\in(3/2,2) $ be given.
Then, the following  statements hold true:
\begin{itemize}
\item[(i)]  {\em (Well-posedness)}  Given $f^{(0)}\in H^{s}(\mathbb{R})$, there exists a unique maximal solution~$(f,v,p)$   to \eqref{STOKES} such that
\begin{itemize}
\item[$\bullet$] $f=f(\cdot;f^{(0)})\in {\rm C}([0,T_+),  H^{s}(\mathbb{R}))\cap {\rm C}^1([0,T_+), H^{s-1}(\mathbb{R})),$
\item[$\bullet$] $v(t)\in {\rm C}^2(\Omega(t))\cap {\rm C}^1(\overline{\Omega(t)})$, $p(t)\in {\rm C}^1(\Omega(t))\cap {\rm C}(\overline{\Omega(t)})$ for all ${t\in(0,T_+)}$,
\item[$\bullet$] $ v(t)|_{\G(t)}\circ\Xi_{f(t)}\in H^2(\R)^2$ for all $t\in(0,T_+)$,
\end{itemize}
where $T_+=T_+( f^{(0)})\in (0,\infty]$ is the maximal existence time.
 Moreover,   the set
 $$\mathcal{M}:=\{(t, f^{(0)})\,:\, f^{(0)}\in H^s(\R),\,0< t<T_+( f^{(0)})\}$$ is open in~${(0,\infty)\times H^s(\R)}$,   
 and $[(t, f^{(0)})\longmapsto f(t; f^{(0)})]$  is a semiflow on $H^s(\R)$ which is smooth in~$\mathcal{M}$.\\[-2.2ex]

\item[(ii)]  {\em (Parabolic smoothing)} 
\begin{itemize}
\item[(iia)]  The map $[(t,\xi)\longmapsto  f(t,\xi)]:(0,T_+)\times\mathbb{R}\longrightarrow\mathbb{R}$ is a ${\rm C}^\infty$-function. \\[-2ex]
\item[(iib)] For any $k\in\N$, we have $f\in {\rm C}^\infty ((0,T_+), H^k(\mathbb{R})).$\\[-2ex]
\end{itemize}

\pagebreak

\item[(iii)]  {\em (Global existence)} If 
$$\sup_{[0,T]\cap [0,T_+( f^{(0)}))} \|f(t)\|_{H^s}<\infty$$
for each $T>0$, then $T_+(f^{(0)})=\infty.$
\end{itemize} 
\end{thm}

 Observe, in particular, that by Theorem~\ref{MT1}~(iib) we have $f(t)\in H^3(\R)$ for~${t>0}$, which justifies the assumptions that were made when deriving the reformulation~\eqref{NNEP1}.
  Thus, the solutions we construct correspond to one-phase Stokes flows, starting from initial domains whose boundaries might have a curvature in distribution sense only.

 Our second main result gives a precise formulation of the limit result announced above. 
 We recall the notation $(f_{\mu^+}(\cdot; f^{(0)}),w^\pm_{\mu^+},q^\pm_{\mu^+})$ for solutions to the two-phase problem \eqref{2STOKES}.
 \begin{thm}\label{MT2} Let  $s\in(3/2,2) $ and $ f^{(0)}\in H^{s}(\mathbb{R})$ be given.
  Let further $(f(\cdot; f^{(0)}),v,p)$  denote the maximal solution to \eqref{STOKES} identified in Theorem~\ref{MT1} and choose $T\in (0,T_+( f^{(0)}))$.
Then,  there exist constants $\varepsilon>0$ and $M>0$ such that for all $\mu^+\in (0, \varepsilon]$, we have $T<T_{+,\mu^+}( f^{(0)})$ and
 \[
\big\|f(\cdot ; f^{(0)})-f_{\mu^+}(\cdot;f^{(0)})\big\|_{{\rm C}([0,T], H^s(\R))}+\Big\|\frac{d}{dt}\big(f(\cdot ; f^{(0)})-f_{\mu^+}(\cdot; f^{(0)})\big)\Big\|_{{\rm C}([0,T], H^{s-1}(\R))}\leq M\mu^+.
  \]
\end{thm}
 
 The proofs of the main results are presented in Section~\ref{Sec:4.5} which concludes the paper.

 \section{The Stokes boundary value problem in a fixed  domain}\label{Sec:2} 

In this section we fix $f\in H^3(\R)$ and we consider the Stokes boundary value problem  
\begin{equation}\label{SBVP}
\left.
\begin{array}{rclll}
\mu\Delta v-\nabla p&=&0&\mbox{in $\Omega$,}\\
\vdiv v&=&0&\mbox{in $\Omega$,}\\{}
T_\mu(v, p)\tilde\nu&=&\sigma\tilde\kappa\tilde\nu&\mbox{on $\Gamma$,}\\
(v, p)(x)&\to&0&\mbox{for $|x|\to\infty$,}
\end{array}\right\}
\end{equation}
where $\0:=\{x=(x_1,x_2)\in\R^2\,:\, x_2<f(x_1)\}$ and $\G:=\{(\xi, f(\xi))\,:\, \xi\in\R\}$.
The main goal is to show that \eqref{SBVP} has a unique solution $(v,p)$, see Theorem~\ref{T:1} below.

We start by introducing  some notation.
Since $\Gamma$ is a  graph over $\R$, it is natural to view $\Gamma$ as is the image of $\R$ under the diffeomorphism $\Xi:=\Xi_f:=({\rm id}_\mathbb{R},f).$
Let now $\nu$ and $\tau$ denote  the componentwise pull-back under~$\Xi$ of the unit normal~$\tilde\nu$ on~$\Gamma$ exterior to~$\Omega$ and of the unit tangent vector $\tilde\tau$ to~$\Gamma$,  i.e.
 \begin{align}\label{nutau}
\nu:=\nu(f):=\frac{1}{\omega}(-f',1)^\top,\qquad\tau:=\tau(f):=\frac{1}{\omega}(1,f')^\top,\qquad \omega:=\omega(f):=(1+f'^2)^{1/2}.
 \end{align}
 Observe that the pull-back~$\kappa:=\omega^{-3}f''\in H^1(\R)$ of the curvature $\tilde\kappa$ satisfies 
 \begin{equation}\label{fundfor}
 \omega\kappa\nu=g',
 \end{equation}
 where  $g:=g(f)$ is given by
\begin{equation}\label{defgn}
  g:=(g_1,g_2)^\top:=(\omega^{-1}-1,\omega^{-1}f')^\top=\Big(-\frac{{f'}^2}{\omega+\omega^2},\frac{f'}{\omega}\Big)^\top. 
\end{equation}

We further recall that  the fundamental solutions $(\cU^k,\cP^k):\mathbb{R}^2\setminus\{0\}\longrightarrow\mathbb{R}^2\times\mathbb{R},$  $k=1,\,2,$
 where $\cU^k=(\cU^k_1, \cU^k_2)^\top$,  to the Stokes equations
\begin{equation}\label{inhstosy}
 \left.\begin{array}{rllll}
\mu \Delta U-\nabla P&=&0,\\[1ex]
\vdiv U&=&0
\end{array}\right\}\qquad\text{in  $\R^2\setminus\{0\}$}
\end{equation}
are given by  
\begin{equation}\label{fundup}
\begin{aligned}
\cU_j^k(y)&=-\frac{1}{4\pi\mu}\left(\delta_{jk}\ln\frac{1}{|y|}+\frac{y_jy_k}{|y|^2}\right),\quad j=1,\,2,\\[1ex]
\cP^k(y)&=-\frac{1}{2\pi}\frac{y_k}{|y|^2} 
\end{aligned}
\end{equation}
for $ y=(y_1,y_2)\in\R^2\setminus\{0\}$, see \cite{Lad63}.

Finally,   defining the mapping $r:=(r^1,r^2):\R^2\longrightarrow\R^2$ by the formula
\begin{equation}\label{rxs}
r:=r(\xi,s):=(\xi-s,f(\xi)-f(s)), \quad (\xi,s)\in\R^2,
\end{equation}
we introduce the double layer potential  $\bD(f)$ for the Stokes equations associated to the hypersurface $\G$ and its $L_2$-adjoint $\bD(f)^\ast$ by the formulas
\begin{equation}\label{defD}
\begin{aligned}
    \bD(f)[\beta](\xi)&:=\frac{1}{\pi}\PV\int_\R\frac{r_1 f'- r_2}{ |r|^4}
   \begin{pmatrix}
r_1^2&r_1 r_2\\
r_1 r_2& r_2^2
\end{pmatrix}\beta\,ds,\\
\bD(f)^\ast[\beta](\xi)&:=\frac{1}{\pi}\PV\int_\R\frac{-r_1 f'(\xi)+ r_2}{|r|^4}
\begin{pmatrix}
r_1^2&r_1 r_2\\
 r_1r_2&r_2^2
\end{pmatrix}\beta\,ds 
\end{aligned}
\end{equation}
for $\beta=(\beta_1,\beta_2)^\top\in L_2(\R)^2$ and $\xi\in\R$. 

 In \eqref{defD}, the integrals are absolutely convergent whenever $f'$ is H\"older continuous. 
 We prefer the definition as principal value integral because we will consider $f\in {\rm BUC}^1(\RRM)$ later. 
Given~${k\in\N}$, ${\rm BUC}^k(\RRM)$ is the Banach space consisting of functions with bounded and uniformly continuous derivatives up to order $k$.
  It is well-known that  the intersection of all these spaces, denoted by ${\rm BUC}^\infty(\RRM)$,
 is a dense subspace of ${\rm BUC}^k(\RRM)$   for each $k\in\N$.
 
Both operators $\bD(f)$, $\bD(f)^*$ can be expressed in terms  of the family~${\{B_{n,m}^0(f)\,:\, n,\, m\in\N\}}$
of singular integral operators introduced in \cite{MBV18,MBV19}, see \eqref{defB0} and~\eqref{DFB} below.
We  now introduce these operators in a more general context.
More precisely, given~${n,\,m\in\N}$ and  Lipschitz continuous  functions~${a_1,\ldots, a_{m},\, b_1, \ldots, b_n:\mathbb{R}\longrightarrow\mathbb{R}}$,  we define
\begin{equation}\label{BNM}
 B_{n,m}(a_1,\ldots, a_m)[b_1,\ldots,b_n,h](\xi):=\frac{1}{\pi}\PV\int_\mathbb{R}  \frac{h(\xi-\eta)}{\eta}\cfrac{\prod_{i=1}^{n}\big(\delta_{[\xi,\eta]} b_i /\eta\big)}{\prod_{i=1}^{m}\big[1+\big(\delta_{[\xi,\eta]}  a_i /\eta\big)^2\big]}\, d\eta,
\end{equation}
where $\delta_{[\xi,\eta]}u:=u(\xi)-u(\xi-\eta)$. 
 For brevity we set
\begin{equation}\label{defB0}
B^0_{n,m}(f):=B_{n,m}(f,\ldots  ,f)[f,\ldots,f,\cdot].
\end{equation}
We note that $B^0_{0,0}=H$, where $H$ denotes the Hilbert transform.

We now prove that  the boundary value problem \eqref{SBVP} is uniquely solvable and that the solution is given by the  hydrodynamic single layer potential with a suitable density function~$\beta$.
\begin{thm}\label{T:1}
 Given $f\in H^3(\R)$, the  boundary value problem \eqref{SBVP} has a  unique solution~$(v,p)$ such that
 \[
v\in {\rm C}^2(\0)\cap {\rm C}^1 (\ov\0),\quad  p\in {\rm C}^1(\0)\cap {\rm C} (\ov\0),\quad   v|_\G \circ\Xi\in H^2(\R)^2.
 \]
 Moreover, letting $\beta=(\beta_1,\beta_2)^\top\in H^2(\R)^2$ denote the unique solution to the equation
  \begin{equation}\label{invertcom}
\Big(\frac{1}{2}- \bD(f)^*\Big)[\beta']=\sigma g',
\end{equation}
where $g\in H^2(\R)^2$ is defined in \eqref{defgn}, we have
\begin{equation}\label{Solutions}
  \begin{aligned}
v(x)&:=\displaystyle\int_\mathbb{R}\partial_s\big(\clu^k(x-(s,f(s)))\big)\beta_k(s)\,ds,\\[1ex]
p( x)&:=\displaystyle-\int_\mathbb{R}\clp^k( x-(s,f(s)))\beta_k'(s)\,ds=\displaystyle\int_\mathbb{R}\p_s\big(\clp^k( x-(s,f(s)))\big)\beta_k(s)\,ds,\quad x\in\Omega.
\end{aligned}
\end{equation}

\end{thm}
\begin{proof}
The unique solvability of  Eq. \eqref{invertcom} is established in Section~\ref{Sec:4} and is taken for granted in this proof. 
We divide  the proof in two steps.\medskip

\noindent{\em Step 1: Uniqueness.}
It suffices to show that the homogeneous boundary value problem \eqref{SBVP} (with the right side of \eqref{SBVP}$_3$ set to be  zero) has only the trivial solution. 
Let thus $(v,p) $ be a solution to the homogeneous system \eqref{SBVP} with   regularity as prescribed above.
We then set~${\0^-:=\0}$, $\0^+:=\R^2\setminus\ov \0$, and we  define~${(w^\pm,q^\pm):\0^\pm\longrightarrow \R^2\times\R}$ by  
\[
(w^-,w^+):=(\mu v,0)\qquad\text{and}  \qquad (q^-,q^+):=(p,0).
\]
Clearly, it holds
\[
w^\pm\in {\rm C}^2(\0^\pm)\cap {\rm C}^1 (\ov{\0^\pm}),\quad  q^\pm\in {\rm C}^1(\0^\pm)\cap {\rm C} (\ov{\0^\pm}),\quad   w^\pm|_\G \circ\Xi\in H^2(\R)^2.
 \]
Moreover, it can be easily checked that $(w^\pm,q^\pm)$ solves the boundary value problem
\be\label{bvpaux}
\left.\begin{array}{rcll}
\Delta w^\pm-\nabla q^\pm&=&0&\mbox{in $\Omega^\pm$,}\\
\vdiv w^\pm&=&0&\mbox{in $\Omega^\pm$,}\\{}
[w]&=&\gamma\circ\Xi^{-1}&\mbox{on $\Gamma$,}\\
{}[T_1(w,q)](\nu\circ \Xi^{-1})&=&0&\mbox{on $\Gamma$,}\\
(w^\pm,q^\pm)(x)&\to&0&\mbox{for $|x|\to\infty$,}
\end{array}\right\}
\ee
with $\gamma:=-w^-|_\Gamma\circ\Xi\in H^2(\R)^2$.
 Given $z^\pm\in  {\rm C} (\ov{\0^\pm}),$ we define $[z]$ as being the jump
 \begin{equation}\label{defjump}
 [z] (x):=z^+(x)-z^-(x),\qquad x\in\Gamma.
\end{equation}
According to \cite[Proposition 2.1]{MP2022}, the system \eqref{bvpaux} has a unique solution.
Moreover, we have  
\begin{equation}\label{eqsupi}
w^\pm|_\Gamma\circ\Xi=\Big(\pm\frac{1}{2}-\bD(f)\Big)[\gamma],
\end{equation}
see \cite[Lemma A.1]{MP2022}.
Since $w^+=0$ and $1/2-\bD(f)\in\kL(L_2(\R)^2)$ is invertible, see~Theorem~\ref{T:L2spec} below, we conclude that $\gamma=0$. 
Consequently $(w^\pm,q^\pm)$ is the trivial solution and this establishes the uniqueness claim.\medskip

\noindent{\em Step 2: Existence.} We are going to verify that $(v,p)$ from \eqref{Solutions} with $\beta$ from \eqref{invertcom} has the announced regularity and satisfies \eqref{SBVP}.
Recalling \eqref{fundup}, we have
  \begin{equation}\label{stokesdiff}
 \begin{aligned}
     \partial_1\clu^1(y)&=\frac{1}{4\pi\mu |y|^4}\begin{pmatrix} y_1(y_1^2-y_2^2)\\[1ex]y_2(y_1^2-y_2^2)\end{pmatrix},&               \partial_1\clp^1(y)&=\frac{y_1^2-y_2^2}{2\pi|y|^4},\\[1ex]
    \partial_2\clu^1(y)&=\frac{1}{4\pi\mu |y|^4}\begin{pmatrix} y_2(y_2^2+3y_1^2)\\[1ex] y_1(y_2^2-y_1^2)\end{pmatrix},&          \partial_2\clp^1(y)&=\frac{2y_1y_2}{ 2\pi|y|^4},\\[1ex]
    \partial_1\clu^2(y)&=\frac{1}{4\pi\mu |y|^4}\begin{pmatrix} y_2(y_1^2-y_2^2)\\[1ex]  y_1(y_1^2+3y_2^2)\end{pmatrix},&         \partial_1\clp^2(y)&=\frac{2y_1y_2}{2\pi|y|^4},\\[1ex]
     \partial_2\clu^2(y)&=\frac{1}{4\pi\mu |y|^4}\begin{pmatrix} y_1(y_2^2-y_1^2)\\[1ex] y_2(y_2^2-y_1^2)\end{pmatrix},&          \partial_2\clp^2(y)&=\frac{y_2^2-y_1^2}{ 2\pi|y|^4}
 \end{aligned}
 \end{equation}
 for $y\neq 0$. 
A direct consequence of \eqref{Solutions} is that $(v,p)$ is defined as an integral of the form
\[(v,p)(x)=\int_\R K(x,s)\beta(s)\,ds,\qquad x\in\0,\]
where, for every $\alpha\in\N^2,$ we have $\partial^\alpha_x K(x,s)=O(s^{-1})$ for $|s|\to\infty$ and locally uniformly  in~${x\in \0}$.
 This shows that $v$ and $p$ are well-defined by \eqref{Solutions}, and that integration and differentiation with respect to $x$ may be interchanged. 
Hence, $(v,p)\in {\rm C}^\infty(\0,\R^2\times\R)$,   and, since~$\p_j(\cU^{k},\cP^{k}),$ $j=1,\, 2$, solve \eqref{inhstosy},  we deduce that $(v,p)$ is 
a solution to \eqref{SBVP}$_1$-\eqref{SBVP}$_2$.

In view of \cite[Lemma A.1]{BM22} it holds that  $p\in {\rm C}(\ov \0)$ and
\begin{equation}\label{tracep}
 p|_\G\circ\Xi=\frac{B_{0,1}^0(f)[\beta_1']+B_{1,1}^0(f)[\beta_2']}{2}-\frac{\beta'\cdot\nu}{2\omega}.
\end{equation}
 Given $\phi\in H^1(\R)$, let~$Z_j[\phi]:\0\longrightarrow\R$, $j=0,\ldots,3,$ be given by
 \[
 Z_j[\phi](x):= \int_\R \frac{(x_1-s)^{3-j}(x_2-f(s))^j}{((x_1-s)^2+(x_2-f(s))^2)^2}\phi(s)\,ds,  \quad x\in\0.
 \]
Since
\[
\p_i v(x)=-\int_\mathbb{R}\partial_i\clu^k(x-(s,f(s)))\beta_k'(s)\,ds,\qquad i=1,\, 2, \, x\in\0,
\]
we obtain, due to \eqref{stokesdiff}, the following formulas
\begin{align*}
&\p_1v_1=-\p_2v_2=-\frac{(Z_0-Z_2)[\beta_1']+(Z_1-Z_3)[\beta_2']}{4\pi\mu },\\[1ex]
&\p_2v_1=-\frac{(Z_3+3Z_1)[\beta_1']+(Z_2-Z_0)[\beta_2']}{4\pi\mu },\\[1ex]
&\p_1v_2=-\frac{(Z_1-Z_3)[\beta_1']+(Z_0+3Z_2)[\beta_2']}{4\pi\mu }.
\end{align*}
Since  $ Z_j[\phi]\in {\rm C}(\overline\Omega)$, see the proof of \cite[Lemma A.1]{MP2021}, we obtain in view of the latter identities that~${v\in {\rm C}^1(\ov \0)}$.
Moreover, the formula derived in the proof of \cite[Lemma A.1]{MP2021} for the traces~${Z_j[\phi]|_\Gamma}$, $j=0,\ldots,3,$ leads us to
\begin{equation}\label{tracenv}
\begin{aligned}
& \p_1v_1|_\G\circ \Xi=-\frac{(B_{0,2}^0(f)-B_{2,2}^0(f))[\beta_1']+(B_{1,2}^0(f)-B_{3,2}^0(f))[\beta_2']}{4\mu}-\frac{f'\beta'\cdot \tau}{2\mu\omega^3},\\[1ex]
&\p_2v_1|_\G\circ \Xi=-\frac{(B_{3,2}^0(f)+3B_{1,2}^0(f))[\beta_1']+(B_{2,2}^0(f)-B_{0,2}^0(f))[\beta_2']}{4\mu }+\frac{\beta'\cdot \tau}{2\mu\omega^3},\\[1ex]
&\p_1v_2|_\G\circ \Xi=-\frac{(B_{1,2}^0(f)-B_{3,2}^0(f))[\beta_1']+(B_{0,2}^0(f)+3B_{2,2}^0(f))[\beta_2']}{4\mu }-\frac{f'^2\beta'\cdot \tau}{2\mu\omega^3}.
\end{aligned}
\end{equation}

It is now a matter of direct computation to infer from  \eqref{defT}, \eqref{fundfor}, \eqref{defD}, \eqref{tracep}, and \eqref{tracenv} 
that the equation $\eqref{SBVP}_3$ is equivalent to \eqref{invertcom},
 hence also  $\eqref{SBVP}_3$ is satisfied.
It remains to check that the far  field boundary condition $\eqref{SBVP}_4$ holds true. 
To this end we infer directly from \cite[Lemma A.4]{BM22} and  \eqref{Solutions}$_2$    that $p$ vanishes at infinity.
Moreover, since by \eqref{Solutions}$_1$ we have
\[
v(x)=\frac{1}{4\pi\mu}
\int_\mathbb{R}\frac{1}{|R|^2}
\left(
\begin{array}{ccc}
-R_2^2&& R_1R_2\\
 R_1R_2&& R_2^2
\end{array}\right)\beta'(s)\,ds
-\frac{1}{4\pi\mu}\int_\mathbb{R}\frac{R_1+f'(s)R_2}{|R|^2}\beta(s)\,ds\]
for $x\in\0$, where $R=(R_1,R_2)$ is given by
\[
 R:=R(s,x):=(x_1-s,x_2-f(s)), \qquad s\in\R,\, x\in\0,
\]
we infer from \cite[Lemma A.4]{BM22} and \cite[Lemma B.2]{MP2021} that also $v$ vanishes at infinity.

In order to show that $v|_\G\circ\Xi\in H^2(\R)^2$ we conclude from \eqref{Solutions}$_1$, \eqref{stokesdiff}, and the formula derived in the proof of \cite[Lemma A.1]{MP2021} for the traces~${Z_j[\phi]|_\Gamma}$, $j=0,\ldots,3,$ that 
\begin{equation}\label{vgamma}
\begin{aligned}
 v_1|_\Gamma\circ\Xi&=\frac{(B_{2,2}^0(f)-B_{0,2}^0(f))[\beta_1-f'\beta_2]-B_{1,2}^0(f)[3f'\beta_1+\beta_2] -B_{3,2}^0(f)[f'\beta_1-\beta_2]}{4\mu},\\[1ex]
v_2|_\Gamma\circ\Xi&=\frac{B_{0,2}^0(f)[f'\beta_1-\beta_2]+(B_{3,2}^0(f)-B_{1,2}^0(f))[\beta_1-f'\beta_2]-B_{2,2}^0(f)[f'\beta_1+3\beta_2]}{4\mu}.
\end{aligned}
\end{equation}
Since $B_{n,m}^0(f)\in\kL(H^2(\R))$, see Lemma~\ref{L:MP0}~(iv) below, we immediately deduce from \eqref{vgamma} that indeed $v|_\G\circ\Xi\in H^2(\R)^2$.
\end{proof}

 \section{On the invertibility  of $\pm1/2+\bD(f)$ and $\pm1/2+\bD(f)^*$}\label{Sec:3}
In this section we establish the invertibility of the operators $\pm1/2-\bD(f)$ and~${\pm1/2-\bD(f)^*}$ in   $\kL(H^k(\R)^2)$, $k=0,\, 1,\, 2$,
 and  $\kL(H^{s-1}(\R)^2)$,~${s\in(3/2,2)},$  under suitable regularity assumptions on $f$.
These properties are needed on the one hand in the proof of Theorem~\ref{T:1}, see~\eqref{invertcom}, and on the other hand when formulating the Stokes flow as an evolution problem for~$f$,
 see~Section~\ref{Sec:4}.
The main step is provided by Theorem~\ref{T:L2spec} below.

\begin{thm}\label{T:L2spec}
Given  $\delta\in(0,1)$, there exists a  constant~$C_0=C_0(\delta)\geq1$ such that for all~${f\in {\rm BUC}^1(\R)}$ with~${\|f'\|_\infty\leq 1/\delta}$ and all $\beta\in L_2(\R)^2$ we have 
\begin{align}\label{DEest}
C_0\min\Big\{\Big\|\Big(\pm\frac{1}{2}-\bD(f)\Big)[\beta]\Big\|_2,\Big\|\Big(\pm\frac{1}{2}-\bD(f)^\ast\Big)[\beta]\Big\|_2\Big\}\geq  \|\beta\|_2.
\end{align}
Moreover, $\pm1/2-\bD(f)^\ast$ and $\pm1/2-\bD(f)$ are invertible in $\kL(L_2(\R))^2$.
\end{thm}

The present section is devoted mainly to the proof of this theorem, which is split up in a number of steps.

\subsection{Preliminaries}

To start, we   reexpress the operators $\bD(f)$ and $\bD(f)^*$ by using the family of singular integral operators defined in \eqref{defB0}  as follows
\begin{equation}\label{DFB}
\begin{aligned}
\bD(f)[\beta] 
&=\begin{pmatrix}
B_{0,2}^0(f)&B_{1,2}^0(f)\\[1ex]
B_{1,2}^0(f)&B_{2,2}^0(f)
\end{pmatrix}
\begin{pmatrix}
f'\beta_1\\[1ex]
f'\beta_2
\end{pmatrix}
-\begin{pmatrix}
B_{1,2}^0(f)&B_{2,2}^0(f)\\[1ex]
B_{2,2}^0(f)&B_{3,2}^0(f)
\end{pmatrix}
\begin{pmatrix}
\beta_1\\[1ex]
\beta_2
\end{pmatrix},\\[1ex]
\bD(f)^*[\beta] 
&=-f'\begin{pmatrix}
B_{0,2}^0(f)&B_{1,2}^0(f)\\[1ex]
B_{1,2}^0(f)&B_{2,2}^0(f)
\end{pmatrix}
\begin{pmatrix}
\beta_1\\[1ex]
\beta_2
\end{pmatrix}
+\begin{pmatrix}
B_{1,2}^0(f)&B_{2,2}^0(f)\\[1ex]
B_{2,2}^0(f)&B_{3,2}^0(f)
\end{pmatrix}
\begin{pmatrix}
\beta_1\\[1ex]
\beta_2
\end{pmatrix}
\end{aligned} 
\end{equation}
for $\beta=(\beta_1,\beta_2)^\top\in L_2(\R)^2$.

Since the operators $B_{n,m}$ are well-studied by now,  mapping properties for the operators~$\bD(f)$ and $\bD(f)^*$  can be obtained by using the representation~\eqref{DFB} and Lemma~\ref{L:MP0} below
 (which collects some important properties of the operators $B_{n,m}$).
  In the following, for $n\in\N$ and Banach spaces $E$ and $F$,   we define $\kL^n_{\rm sym}(E,F)$ as the Banach space of $n$-linear, bounded, and symmetric maps $A: E^n\to F$.
 Moreover, ${\rm C}^{1-}(E,F)$ (resp.~${{\rm C}^{\infty}(E,F)}$) is the space of locally Lipschitz continuous (resp. smooth) mappings from~$E$ to~$F$.
 
\begin{lemma}\label{L:MP0}\,

\begin{itemize}
\item[(i)] Given  Lipschitz continuous  functions $a_1,\ldots, a_{m},\, b_1, \ldots, b_n:\mathbb{R}\longrightarrow\mathbb{R}$,  there exists a constant~$C$ depending only 
on $n,\, m$ and $\max_{i=1,\ldots, m}\|a_i'\|_{\infty}$, such that
 $$\|B_{n,m}(a_1,\ldots, a_m)[b_1,\ldots,b_n,\,\cdot\,]\|_{\kL(L_2(\mathbb{R}))}\leq C\prod_{i=1}^{n} \|b_i'\|_{\infty}.$$ 
 Moreover,   $B_{n,m}\in {\rm C}^{1-}((W^1_\infty(\mathbb{R}))^{m},\kL^n_{\rm sym}(W^1_\infty(\mathbb{R}), \kL(L_2(\mathbb{R})))).$
  \item[(ii)] Let   $n\geq1,$  $s\in(3/2,2),$  and  $a_1,\ldots, a_m\in H^s(\R)$ be given.  
  Then, there exists a constant~$C$, depending only on $n,\, m$, $s$,  and $\max_{1\leq i\leq m}\|a_i\|_{H^s}$, such that
\begin{align} 
&\| B_{n,m}(a_1,\ldots, a_{m})[b_1,\ldots, b_n,h]\|_2\leq C\|b_1\|_{H^1}\|h\|_{H^{s-1}}\prod_{i=2}^{n}\|b_i\|_{H^s} \label{REF1}
\end{align}
for all $b_1,\ldots, b_n\in H^s(\R)$ and $h\in H^{s-1}(\R).$
 \item[(iii)] Given $s\in(3/2 ,2)$ and $a_1,\ldots, a_m, b_1,\ldots, b_n\in H^s(\mathbb{R})$, there exists a constant $C,$
  depending only on $n,\, m,\, s$,  and $\max_{1\leq i\leq m}\|a_i\|_{H^s},$ such that
\begin{align*} 
\| B_{n,m}(a_1,\ldots, a_{m})[b_1,\ldots, b_n,\cdot]\|_{\kL(H^{s-1}(\mathbb{R}))}\leq C \prod_{i=1}^{n}\|b_i\|_{H^{s}}.
\end{align*}
Moreover,   $  B_{n,m}\in {\rm C}^{1-}((H^s(\mathbb{R}))^m, \kL^n_{\rm sym}(H^s(\mathbb{R}), \kL(H^{s-1}(\mathbb{R})))).$ \\[-1ex]

 \item[(iv)] Let   $a_1,\ldots, a_m\in H^2(\R)$   be given. 
  Then, there exists a constant~$C$, depending  only on~$n,\, m$,  and $\max_{1\leq i\leq m}\|a_i\|_{H^2}$, such that
\begin{align} 
\| B_{n,m}(a_1,\ldots, a_{m})[b_1,\ldots, b_n,h]\|_{H^1}\leq C \|h\|_{H^1}\prod_{i=1}^{n}\|b_i\|_{H^2} \label{FER1}
\end{align}
for all $b_1,\ldots, b_n\in H^2(\R)$ and $h\in H^1(\R) $, with
\begin{equation}\label{FDER}
\begin{aligned}
&\hspace{-1cm}(B_{n,m}(a_1,\ldots, a_{m})[b_1,\ldots, b_n,h])'\\[1ex]
&=B_{n,m}(  a_1,\ldots,  a_m) [b_1,\ldots , b_n, h' ]\\[1ex]
&\hspace{0,45cm}+\sum_{i=1}^nB_{n,m}(a_1,\ldots,a_m)[b_1,\ldots,b_{i-1}, b_i',b_{i+1},\ldots  b_n, h]\\[1ex]
&\hspace{0,45cm}-2\sum_{i=1}^mB_{n+2,m+1}( a_1,\ldots,  a_i, a_i,\ldots,a_m) [b_1,\ldots,b_n, a_i',a_i, h ].
\end{aligned}
\end{equation}
Moreover, $B_{n,m}\in {\rm C}^{1-}(H^2(\R)^m,\kL^{n}_{\rm sym}(H^2(\R),\kL(H^1(\R)))).$ 
\end{itemize}
\end{lemma}
\begin{proof}
The claims (i) and (ii) are established in \cite[Lemmas~3.1 and~3.2]{MBV18},   the property~(iii) is proven in \cite[Lemmas~5]{AM22}, 
and~(iv) is established in \cite[Lemma 4.3]{MP2022}.
\end{proof}

As a direct consequence of \eqref{DFB} and Lemma~\ref{L:MP0}~(i) we obtain that 
\begin{equation}\label{CDD*}
[f\longmapsto \bD(f)],\, [f\longmapsto \bD(f)^*]\in {\rm C}^{1-}(W^1_\infty(\R),\kL(L_2(\R)^2)).
\end{equation}

Moreover,  by Lemma~\ref{L:MP0}~(iv), we have
\begin{equation}\label{Bnmk}
 [f\longmapsto B_{n,m}^0(f)]\in {\rm C}^{1-}(H^{k+1}(\R),\kL(H^k(\R)))\qquad\text{for $k=1,\, 2$.} 
\end{equation}

\subsection{Rellich identities on $\Gamma$}

The proof of Theorem~\ref{T:L2spec} relies on several Rellich identities for the Stokes problem, \eqref{RELLICH1}-\eqref{RELLICH5} below, 
which hold also in a bounded geometry~in $\R^n$,~$n\geq3$,
see~\cite{FKV88}. 

 Let $f\in{\rm BUC}^\infty(\RRM)$ and $\beta=(\beta_1,\beta_2)^\top\in {\rm C}_c^\infty(\RRM)^2$.
Using the notation from Section~\ref{Sec:2}, we set~${\0^-:=\0}$, $\0^+:=\R^2\setminus\ov\0$, and we define the hydrodynamic single-layer potential $(u, \Pi)$ by  the formula
\begin{equation}\label{iupi}
\left.\begin{aligned}
 u(x)&:=u(f)[\beta](x):=-\int_\R\cU^k( x-(s,f(s)))\beta_k(s)\, ds\\[1ex]
 \Pi(x)&:= \Pi(f)[\beta](x):=-\int_\R\cP^k( x-(s,f(s)))\beta_k(s)\, ds
\end{aligned} 
\right\} \qquad\text{for $x\in\R^2\setminus\Gamma$,}
 \end{equation}
 where  $\cU^k$ and $\cP^k$ are defined in \eqref{fundup} (with $\mu=1$). 
Since  $\beta$ is  compactly supported, it is not difficult to see that the functions $(u,\Pi)$ are well-defined, smooth in~$\Omega^\pm$, and satisfy
\begin{equation}\label{Sto34}
 \left.
 \begin{array}{rllll}
\Delta u-\nabla \Pi&=&0,\\[1ex]
\vdiv u&=&0
\end{array}\right\}\qquad\text{in  $\0^\pm$},
\end{equation}
as well as
\begin{equation}\label{decay}
\Pi,\,\nabla u=O(|x|^{-1})\qquad \mbox{for $|x|\to\infty$}.
\end{equation}
 Moreover,  \cite[Lemma~A.1]{BM22} and  the arguments in the proof of \cite[Lemma~A.1]{MP2021} enable us to conclude that $ \Pi^\pm:=\Pi|_{\Omega^\pm}$ and $u^\pm:=u|_{\Omega^\pm}$ satisfy
  $\Pi^\pm\in {\rm C}(\overline{\Omega^\pm})$ and 
$u^\pm\in {\rm C}^1(\overline{\Omega^\pm})$, with
\begin{equation}\label{repp}
\begin{aligned}
     \partial_iu_j^\pm|_\G \circ\Xi(\xi) &=-\PV\int_\R\partial_i\cU_j^k(r)\beta_k\,ds\pm\frac{-\beta_j\nu^i+\nu^i\nu^j\beta\cdot\nu}{2\omega}(\xi),\quad i,\, j=1,\, 2,\\
     \Pi^\pm|_\G\circ\Xi(\xi) &= -\PV\int_\R\cP^k(r)\beta_k\,ds\pm\frac{\beta\cdot\nu}{2\omega}(\xi)
\end{aligned}
\end{equation}
for $\xi\in\RRM$, with $r=r(\xi,s)$  defined in  \eqref{rxs}.

Recalling the definition \eqref{defT} of the stress tensor, we then compute in view of \eqref{Sto34}
\begin{equation*}
\begin{aligned}
{\rm div\, }\begin{pmatrix}
0\\[1ex]
\|\nabla u+(\nabla u)^\top\|_F^2
\end{pmatrix}&=4 \,{\rm div\, } \big(T_1(u,\Pi)\p_2 u\big)\qquad\text{in $\R^2\setminus\Gamma$},\\[1ex]
{\rm div\,}\begin{pmatrix}
0\\[1ex]
\|\nabla u\|_F^2
\end{pmatrix}&
=2\, {\rm div\, } \big(((\nabla u)^\top -\Pi E_2)\p_2 u\big)\qquad\text{in $\R^2\setminus\Gamma$},\\[1ex]
{\rm div\,}\begin{pmatrix}
0\\[1ex]
\Pi^2
\end{pmatrix}&=2\,{\rm div\,}\big( (\p_1u_2-\p_2u_1)\p_1u+ \Pi\nabla u_2\big)\qquad\text{in $\R^2\setminus\Gamma$,}
\end{aligned}
\end{equation*}
where $\|\cdot\|_F$ denotes as usual the Frobenius norm of matrices.
 Using \eqref{decay}, we may integrate the latter identities over $\0^\pm$ to obtain, in view of Gauss' theorem, the Rellich identities
\begin{align} 
\int_\Gamma \|\nabla u^\pm+(\nabla u^\pm)^\top\|_F^2\tilde\nu^2\, d\Gamma&=4\int_\Gamma\p_2 u^\pm  \cdot  T_1(u,\Pi)^\pm  \tilde \nu\, d\Gamma,\label{RELLICH1}\\[1ex]
\int_\Gamma \|\nabla u^\pm\|_F^2\tilde\nu^2\, d\Gamma&=2\int_\Gamma \p_2 u^\pm \cdot ( \nabla u^\pm -\Pi^\pm E_2)\tilde\nu\, d\Gamma,\label{RELLICH2}\\[1ex]
\int_\Gamma |\Pi^\pm|^2\tilde\nu^2\, d\Gamma&=2\int_\Gamma (\p_1u_2^\pm-\p_2u_1^\pm)\p_{\tilde\tau}u_2^\pm+\Pi^\pm \p_{\tilde\nu}u_2^\pm\, d\Gamma.\label{RELLICH3}
\end{align}
We now subtract \eqref{RELLICH1} from  \eqref{RELLICH2} multiplied by $4$ to get
  \begin{align} 
 \int_\Gamma |\p_1u_2^\pm-\p_2u_1^\pm|^2\tilde\nu^2\, d\Gamma=2\int_\Gamma \Pi^\pm\p_{\tilde\tau} u_1^\pm-(\p_1u_2^\pm-\p_2u_1^\pm) \p_{\tilde\nu} u_1^\pm\, d\Gamma.\label{RELLICH4}
\end{align}
Furthermore, from
\[
(T_1(u^\pm,\Pi^\pm)+\Pi^\pm E_2)\tilde\nu=\big(\nabla u^\pm+(\nabla u^\pm)^\top\big)\tilde\nu=\begin{pmatrix}
\p_{\tilde \nu}u_1^\pm+\p_{\tilde\tau} u_2^\pm\\[1ex]
\p_{\tilde \nu}u_2^\pm-\p_{\tilde\tau} u_1^\pm
\end{pmatrix},
\]
we obtain, after taking the difference of \eqref{RELLICH3} and \eqref{RELLICH4}, 
\begin{equation}\label{RELLICH5}
\begin{aligned} 
&{\hspace{-1cm}} \int_\Gamma |\p_1u_2^\pm-\p_2u_1^\pm|^2\tilde\nu^2\, d\Gamma\\[1ex]
&=\int_\Gamma |\Pi^\pm|^2\tilde\nu^2\, d\Gamma
 -2\int_\Gamma 
 \begin{pmatrix}
 \p_1u_2^\pm-\p_2u_1^\pm\\[1ex]
 \Pi^\pm
 \end{pmatrix}\cdot \big(T_1(u^\pm,\Pi^\pm)+\Pi^\pm E_2\big)\tilde\nu\, d\Gamma.
\end{aligned}
\end{equation}

\subsection{Transformation to the real axis}\ 
To represent the pull-backs of the one-sided traces of $(\nabla u)_{ij}$  and $\Pi$  on $\G$ as singular integral operators, 
we define, for~$f\in W^1_\infty(\RRM)$, ${\beta\in   L_2}(\RRM)^2$, 
 and $\theta\in L_2(\R)$, the singular integral operators $\bT_i$, $\bB_i$, $i=1,\,2$, as follows:
\begin{align*}
 \bT_1(f)[\beta](\xi)&:=\frac{1}{4\pi}\PV\int_\R\frac{1}{|r|^4}
\begin{pmatrix}
 r_1r_2^2-r_1^3& r_2^3- r_1^2 r_2\\
r_2^3-r_1^2 r_2&  -r_1^3 -3r_1r_2^2
 \end{pmatrix}\beta\,ds,\\[1ex]
 \bT_2(f)[\beta](\xi)&:=\frac{1}{4\pi}\PV\int_\R\frac{1}{|r|^4}
\begin{pmatrix}
 - r_2^3-3r_1^2r_2& r_1^3- r_1 r_2^2\\
r_1^3-r_1 r_2^2&  r_1^2 r_2- r_2^3
 \end{pmatrix}\beta\,ds,\\[1ex]
 \bB_1(f)[\theta](\xi)&:=\frac{1}{\pi}\PV\int_\R\frac{- r_1 f'+ r_2}{ |r|^2}\,\theta\,ds,\\[1ex]
    \bB_2(f)[\theta](\xi)&:=\frac{1}{\pi}\PV\int_\R\frac{r_1 + r_2 f'}{|r|^2}\,\theta\,ds,
 \end{align*}
in the notation introduced in~\eqref{rxs}. Since the components of these operators may be expressed by using only the singular operators~${B^{0}_{n,m}(f)}$, we infer from Lemma~\ref{L:MP0}~(i) that 
 \begin{equation}\label{regTB}
 \bT_i\in {\rm C}^{1-}(W^1_\infty(\R),\kL(L_2(\R)^2)),\quad \bB_i\in {\rm C}^{1-}(W^1_\infty(\R),\kL(L_2(\R))).
 \end{equation}

It follows from \eqref{fundup}, \eqref{stokesdiff} (with~${\mu=1}$) and~\eqref{repp} that for~${f\in{\rm BUC}^\infty(\RRM)}$ and $\beta\in {\rm C}_c^\infty(\RRM)^2$
we have (in matrix notation)
\be\label{identitiesR}
\begin{array}{rl}
 \nabla u^\pm(f)[\beta]\big|_\G\circ\Xi&=\Big(\bT_1(f)[\beta]\;\bT_2(f)[\beta]\Big)\mp\cfrac{(\beta\cdot\tau)}{2\omega}\tau\,\nu^\top
=:\wtdu(f)[\beta],\\[2ex]
 \Pi^\pm(f)[\beta]\big|_\G\circ\Xi&=\displaystyle\frac{(\pm1+\bB_1(f))[\omega^{-1}\beta\cdot\nu]+\bB_2(f)[\omega^{-1}\beta\cdot\tau]}{2}
=:\wtp(f)[\beta],
\end{array}
\ee
the right sides of \eqref{identitiesR} being meaningful whenever $f\in W^1_\infty(\RRM)$ and $\beta\in   L_2(\RRM)^2$.
To translate the Rellich identities of the previous subsection to identities for integral operators on $\R$ it is convenient to additionally introduce the operators $\wtt$ and $\wtdtwou$ by
\begin{align*}
  \wtt(f)[\beta]&:=\wtdu(f)[\beta]+\wtdu(f)[\beta]^\top-\wtp(f)[\beta] E_2,\\
  \wtdtwou(f)[\beta]&:=\wtdu(f)[\beta]e_2, 
  \end{align*}
where $e_2:=(0,1)^\top$.

From \eqref{regTB} we immediately get
\be\label{regpullback}
\left.\begin{array}{rl}
\wtdu,\wtt&\in {\rm C}^{1-}(W^1_\infty(\R),\kL(L_2(\R)^2,L_2(\R)^{2\times 2})),\\
\wtdtwou&\in {\rm C}^{1-}(W^1_\infty(\R),\kL(L_2(\R)^2,L_2(\R)^2)),\\
\wtp&\in {\rm C}^{1-}(W^1_\infty(\R),\kL(L_2(\R)^2,L_2(\R))).
\end{array}\right\}
\ee

  It is not difficult to check that
 \begin{equation}\label{ffff}
 \omega \wtt(f)[\beta]\nu=\Big(\mp\frac{1}{2}-\bD(f)^\ast\Big)[\beta].
 \end{equation}
 
 Parameterizing $\Gamma$  over $\R$ via $[s\longmapsto(s,f(s))]$
   and using \eqref{identitiesR} and \eqref{ffff}, we find from~\eqref{RELLICH1} that
 \begin{equation}\label{RELLICH1'}
\Big\|\wtt(f)[\beta]+\wtp(f)[\beta]E_2\Big\|_2^2=4\Big\langle\wtdtwou(f)[\beta]\,\Big|\,\Big(\mp\frac{1}{2}-\bD(f)^\ast\Big)[\beta]\Big\rangle_2,
\end{equation}
where $\langle\cdot|\cdot\rangle$ denotes the $L_2(\R)^2$ scalar product.
Similarly, from \eqref{RELLICH2} and \eqref{RELLICH5} we get 
 \begin{equation} \label{RELLICH2'}
\Big\|\wtdu(f)[\beta]\Big\|_2^2=2 \Big\langle\wtdtwou(f)[\beta] \Big| \omega^{-1}\Big( \wtdu(f)[\beta] -\wtp(f)[\beta] E_2\Big)\nu  \Big\rangle_2
\end{equation}
and
\begin{equation}\label{RELLICH5'}
\begin{aligned} 
&{\hspace{-1cm}}2\left\langle
 \begin{pmatrix}
 \Big(\widetilde{(\nabla u)^\pm_{21}}(f)-\widetilde{(\nabla u)^\pm_{12}}(f)\Big)[\beta]\\[1ex]
 \wtp(f)[\beta]
 \end{pmatrix} \right| \left.\Big(\mp\frac{1}{2}-\bD(f)^\ast\Big)[\beta]+\wtp(f)[\beta] \begin{pmatrix}
 -f'\\[1ex]
 1
 \end{pmatrix} \right\rangle_2\\[1ex]
 &=\Big\|\wtp(f)[\beta]\Big\|_2^2 -\Big\|\Big(\widetilde{(\nabla u)^\pm_{21}}(f)-\widetilde{(\nabla u)^\pm_{12}}(f)\Big)[\beta]\Big\|_2^2,
\end{aligned}
\end{equation}
 respectively. 
 
 By a standard density argument, it follows from \eqref{CDD*} and \eqref{regpullback} that \eqref{RELLICH1'}--\eqref{RELLICH5'} 
 hold for any~${f\in{\rm  BUC} ^1(\RRM)}$ and~$\beta\in  L_2(\R)^2$.
 
 \subsection{Completion of the proof of Theorem \ref{T:L2spec}} We divide the remaining arguments in the proof of Theorem \ref{T:L2spec} in three steps.\medskip

\noindent{\em Step 1.} Fix $\delta\in (0,1)$ and ${f\in{\rm BUC}^1(\RRM)}$  such that $\|f'\|_\infty\leq 1/\delta$. 
  In the sequel, we are going to write $C(\delta)$ for different positive constants that depend on $\delta$ only.
 Let $\beta\in L_2(\R)^2$.
 Using Lemma~\ref{L:MP0}~(i), we find  a constant  $C(\delta)$ such that  the right side of \eqref{RELLICH1'} satisfies
 \begin{equation*}
4\Big\langle\wtdtwou(f)[\beta]\,\Big|\,\Big(\mp\frac{1}{2}-\bD(f)^\ast\Big)[\beta]\Big\rangle_2\leq C(\delta)\Big\|\Big(\mp\frac{1}{2}-\bD(f)^\ast\Big)[\beta]\Big\|_2\|\beta\|_2.
\end{equation*}
For the left side of \eqref{RELLICH1'} we have, in view of  \eqref{ffff} and Lemma~\ref{L:MP0}~(i),
\begin{align*}
\Big\|\wtt(f)[\beta]+\wtp(f)[\beta]E_2\Big\|_2^2\geq& \frac{\delta^2}{2}\int_\R \Big|\omega \wtt(f)[\beta]\nu+\omega\wtp(f)[\beta]\nu\Big|^2 \, dx\\[1ex]
\geq&\frac{\delta^2}{4}\Big\|\wtp(f)[\beta]\Big\|_2^2-\frac{\delta^2}{2}\Big\|\Big(\mp\frac{1}{2}-\bD(f)^\ast\Big)[\beta]\Big\|_2^2\\[1ex]
\geq&\frac{\delta^2}{4}\Big\|\wtp(f)[\beta]\Big\|_2^2-C(\delta)\Big\|\Big(\mp\frac{1}{2}-\bD(f)^\ast\Big)[\beta]\Big\|_2\|\beta\|_2.
\end{align*}
These estimates show that there exists a constant $ C(\delta)$ with the property that
\begin{equation}\label{L32}
C(\delta)\Big\|\Big(\mp\frac{1}{2}-\bD(f)^\ast\Big)[\beta]\Big\|_2\|\beta\|_2\geq\Big\|\wtp(f)[\beta]\Big\|_2^2
\end{equation}
 for all $\beta\in L_2(\R)^2$.\medskip

\noindent
{\em Step 2.} It follows from \eqref{RELLICH5'} that
\begin{equation*} 
 \begin{aligned}
 &\Big\|\Big(\widetilde{(\nabla u)^\pm_{21}}(f)-\widetilde{(\nabla u)^\pm_{12}}(f)\Big)[\beta]\Big\|_2^2\\[1ex]
 &\leq C(\delta)\bigg[\Big\|\Big(\mp\frac{1}{2}-\bD(f)^\ast\Big)[\beta]\Big\|_2^2+\Big\|\wtp(f)[\beta]\Big\|_2^2\\[1ex]
 &\hspace{1.5cm}+\Big\|\Big(\widetilde{(\nabla u)^\pm_{21}}(f)-\widetilde{(\nabla u)^\pm_{12}}(f)\Big)[\beta]\Big\|_2
 \Big(\Big\|\Big(\mp\frac{1}{2}-\bD(f)^\ast\Big)[\beta]\Big\|_2+\Big\|\wtp(f)[\beta]\Big\|_2\Big)\bigg],
 \end{aligned}
 \end{equation*}
 hence 
 \begin{equation}\label{estrot}
\Big\|\Big(\widetilde{(\nabla u)^\pm_{21}}(f)-\widetilde{(\nabla u)^\pm_{12}}(f)\Big)[\beta]\Big\|_2^2\leq C(\delta)
\Big(\Big\|\Big(\mp\frac{1}{2}-\bD(f)^\ast\Big)[\beta]\Big\|_2^2+\Big\|\wtp(f)[\beta]\Big\|_2^2\Big).
 \end{equation}
 Furthermore, as
 \[
 2\wtdu(f)[\beta]\nu=\frac{1}{\omega}\Big(\mp\frac{1}{2}-\bD(f)^\ast\Big)[\beta]
 +\wtp(f)[\beta]\nu-\Big(\widetilde{(\nabla u)^\pm_{21}}(f)-\widetilde{(\nabla u)^\pm_{12}}(f)\Big)[\beta]\tau,
 \]
 we infer from \eqref{estrot} that
 \begin{equation}\label{estjacnorm}
\Big\|\wtdu(f)[\beta]\nu\Big\|_2^2\leq C(\delta)
\Big(\Big\|\Big(\mp\frac{1}{2}-\bD(f)^\ast\Big)[\beta]\Big\|_2^2
 +\Big\|\wtp(f)[\beta]\Big\|_2^2\Big).
 \end{equation}
 The identity \eqref{RELLICH2'}  implies the estimate
 \[\Big\|\wtdu(f)[\beta]\Big\|_2^2\leq C(\delta)\Big\|\wtdu(f)[\beta]\Big\|_2
 \Big(\Big\|\wtdu(f)[\beta]\nu\Big\|_2+\Big\|\wtp(f)[\beta]\Big\|_2\Big),\]
 and together with \eqref{estjacnorm} this yields
 \begin{equation}\label{estjac}
 \Big\|\wtdu(f)[\beta]\Big\|_2^2\leq C(\delta)
\Big(\Big\|\Big(\mp\frac{1}{2}-\bD(f)^\ast\Big)[\beta]\Big\|_2^2
 +\Big\|\wtp(f)[\beta]\Big\|_2^2\Big).
 \end{equation}
 Multiplying the identity \eqref{identitiesR}$_1$ by $e_2$ and taking subsequently  the scalar product with $\beta$, we observe that
 \[\|\beta\cdot\tau\|_2^2=\mp2\omega^2\Big(\Big\langle\beta\,\Big|\,\wtdu(f)[\beta]e_2\Big\rangle
 -\langle\beta\,|\,\bT_2(f)[\beta]\rangle\Big).\]
 The second term on the right vanishes as $\bT_2(f)^\ast=-\bT_2(f)$, and thus
 \begin{equation}\label{estbtan}
 \|\beta\cdot\tau\|_2^2\leq C(\delta)\Big\|\wtdu(f)[\beta]\Big\|_2\|\beta\|_2
 \end{equation}
Next, we rewrite \eqref{identitiesR}$_2$ as
\[(\pm1+\bB_1(f)[\omega^{-1}\beta\cdot\nu]=2\wtp(f)[\beta]-\bB_2(f)[\omega^{-1}\beta\cdot\tau].\]
Letting $\bA(f):=\bB_1(f)^\ast$, it follows from  the Rellich identity for the Muskat problem established in the proof of~\cite[Theorem~3.5]{MBV18} that the operator~${(\pm 1-\bA(f))\in\kL(L_2(\R))}$  is an isomorphism with
\[\|\left(\pm 1+\bA(f)\right)^{-1}\|_{\kL(L_2(\R))}\leq C(\delta).\]
This implies that also its adjoint $(\pm 1- \bB_1(f))\in\kL(L_2(\R^2))$ is an isomorphism  and
\[ \|(\pm 1+\bB_1(f))^{-1}\|_{ \kL(L_2(\R))}\leq C(\delta).\]
Using this and Lemma~\ref{L:MP0}~(i) we get
\[\|\beta\cdot\nu\|_2\leq C(\delta)\Big(\|\beta\cdot\tau\|_2+\Big\|\wtp(f)[\beta]\Big\|_2\Big),\]
and together with \eqref{estbtan} and Young's inequality we arrive at
\[\|\beta\|_2^2=\|\beta\cdot\nu\|_2^2+\|\beta\cdot\tau\|_2^2\leq C(\delta)\Big(\Big\|\wtdu(f)[\beta]\Big\|_2^2+\Big\|\wtp(f)[\beta]\Big\|_2^2\Big).\]
In view  of \eqref{estjac} we infer from the latter inequality that
\[\|\beta\|_2^2\leq C(\delta)\Big(\Big\|\Big(\mp\frac{1}{2}-\bD(f)^\ast\Big)[\beta]\Big\|_2^2+\Big\|\wtp(f)[\beta]\Big\|_2^2\Big),\]
and together with \eqref{L32} and Young's inequality we finally obtain
 \begin{equation}\label{DEest1}
 \|\beta\|_2\leq C(\delta)\Big\|\Big(\mp\frac{1}{2}-\bD(f)^\ast\Big)[\beta]\Big\|_2.
 \end{equation}
 
\noindent{\em Step 3.} In view of the identity
\[
(\lambda-\bD(f)^\ast)[\beta]=\Big(\mp\frac{1}{2}-\bD(f)^\ast\Big)[\beta]+\Big(\lambda\pm\frac{1}{2}\Big)\beta, \quad     \lambda\in\mathbb C,\, \beta\in L_2(\R)^2,
\]
we deduce from \eqref{DEest1} that 
 \[C(\delta)\|(\lambda-\bD(f)^\ast)[\beta]\|_2\geq (1-C(\delta)|\lambda\pm1/2|)\|\beta\|_2 , \quad     \lambda\in\mathbb C,\, \beta\in L_2(\R)^2,\]
and therefore
 \begin{equation}\label{DEest2}
 \|\beta\|_2\leq  C(\delta)\|(\lambda-\bD(f)^\ast)[\beta]\|_2\quad
 \text{for $\lambda$ sufficiently close to $\pm1/2$ and    $ \beta\in L_2(\R)^2$.}
 \end{equation}
  Now \eqref{DEest2} together with the estimate \cite[(3.15)]{MP2022} shows there exists a constant~${C_0=C_0(\delta)\geq 1}$ such that 
\[
C_0\|(\lambda-\bD(f)^*)[\beta]\|_2\geq \|\beta\|_2\qquad\text{for all $\beta\in L_2(\R)^2$ and all $\lambda\in\R\setminus(-1/2,1/2)$.}
\]
 As $\bD(f)^\ast$ is in $\kL(L_2(\R)^2),$ the shift $\lambda-\bD(f)^\ast\in\kL(L_2(\R)^2)$ is an isomorphism if~$|\lambda|$ is sufficiently large.
The method of continuity, cf. e.g. \cite[Proposition~I.1.1.1]{Am95}, implies now that~${\pm1/2-\bD(f)^*}$, and hence also $\pm1/2-\bD(f)$, are  isomorphisms as well.
 This completes the proof of Theorem \ref{T:L2spec}.
 
 \subsection{Spectral properties in Sobolev spaces}
In Lemma~\ref{L:3} we establish the invertibility of the operators considered in Theorem~\ref{T:L2spec} in the Banach algebras $\kL(H^k(\R)^2)$, $k=1,\, 2$.
\begin{lemma}\label{L:3} \ 
 For $f\in H^{k+1}(\R)$, $k=1,\, 2$, the operators~${\pm1/2-\bD(f)}$ and~${\pm1/2-\bD(f)^*}$ are invertible in $\kL(H^k(\R)^2)$.
\end{lemma}
\begin{proof}
Fix $f\in H^{k+1}(\R)$. 
The representation \eqref{DFB} and Lemma~\ref{L:MP0}~(iv) then immediately imply that~$\bD(f)$ and $\bD(f)^*$ belong to $\kL(H^k(\R)^2)$.

Let first $k=1$.  Using \eqref{FDER}, we compute that the components of 
\begin{align*}
T[\beta]:=(\bD(f)[\beta])'-\bD(f)[\beta'], \qquad\beta=(\beta_1,\beta_2)^\top\in H^1(\R)^2,
\end{align*}
  are (finite) linear combination of terms of the form
\begin{align*}
B_{n,m}(f,\ldots,f)[f',f,\ldots,f,f'^{\ell}\beta_i]\quad\text{and} \qquad B_{n,m}^0(f)[f''\beta_i]
\end{align*}
with $n,\,m\leq 5$, $\ell=0,\, 1,$ and $i=1,\, 2$.
Choosing $s\in(3/2,2)$, it follows from Lemma~\ref{L:MP0}~(i)-(ii) that there exists a constant $C_1>0$ such that
\[
\|T[\beta]\|_2\leq C_1\|\beta\|_{H^{s-1}}, \qquad\beta=(\beta_1,\beta_2)^\top\in H^1(\R)^2.
\]
This property together with \eqref{DEest} now leads to
\begin{align*}
\|(\pm1/2-\bD(f))[\beta]\|_{H^1}^2&=\|(\pm1/2-\bD(f))[\beta]\|_{2}^2+\|((\pm1/2-\bD(f))[\beta])'\|_{2}^2\\[1ex]
&=\|(\pm1/2-\bD(f))[\beta]\|_{2}^2+\frac{1}{2}\|(\pm1/2-\bD(f))[\beta']\|_{2}^2-\|T[\beta]\|_2^2\\[1ex]
&\geq \frac{1}{2C_0^2} \|\beta\|_{H^1}^2 - C_1^2\|\beta\|_{H^{s-1}}^2.
\end{align*}
The latter estimate, an interpolation argument, and Young's inequality imply  there exists a further constant~${C_2=C_2(\delta)\geq1}$ such that
 \begin{align*}
C_2\big(\|\beta\|_2^2+\|(\pm1/2-\bD(f))[\beta]\|_{H^1}^2\big)\geq  \|\beta\|_{H^1}^2  
\end{align*}
for all $\beta\in H^1(\R)^2$.
This estimate combined with \eqref{DEest} now yields
 \begin{align*}
C_2(C_0^2+1)\|(\pm1/2-\bD(f))[\beta]\|_{H^1}^2\geq  \|\beta\|_{H^1}^2  
\end{align*}
for all $\beta\in H^1(\R)^2$. 
The invertibility of $\pm1/2-\bD(f)$ in $\kL(H^1(\R)^2)$ follows from this estimate and  the invertibility property in $\kL(L_2(\R)^2)$.
The invertibility of  $\pm1/2-\bD(f)^*$ in $\kL(H^1(\R)^2)$ may be established by using the same arguments and therefore we omit  the details.

Finally, when $k=2$, the invertibility of $\pm1/2-\bD(f)$ and $\pm1/2-\bD(f)^*$ in $\kL(H^2(\R)^2)$ may be obtained by arguing along the same lines as above
 (see the proof of \cite[Theorem 4.5]{MP2022} for some details).
\end{proof}

The next invertibility result is used in Section~\ref{Sec:4.4}  when we consider our evolution problems
in~${H^{s-1}(\R)}$ with~${s\in(3/2,2)}$. 
\begin{lemma}\label{L:4}
Given  $\delta\in(0,1)$ and $s\in(3/2,2),$ there exists a positive constant~${C=C(\delta,s)\geq 1}$ such that 
for all~${f\in H^s(\R)}$ with~${\|f \|_{H^s}\leq 1/\delta}$ and all $\beta\in H^{s-1}(\R)^2$ we have
\begin{align}\label{DEests}
C\min\Big\{\Big\|\Big(\pm\frac{1}{2}-\bD(f)\Big)[\beta]\Big\|_{H^{s-1}},\Big\|\Big(\pm\frac{1}{2}-\bD(f)^\ast\Big)[\beta]\Big\|_{H^{s-1}}\Big\}\geq  \|\beta\|_{H^{s-1}}.
\end{align}

Moreover, $\pm1/2-\bD(f)^\ast$ and $\pm1/2-\bD(f)$ are invertible in $\kL(H^{s-1}(\R)^2)$.
\end{lemma}
\begin{proof}
As a direct consequence of Lemma~\ref{L:MP0}~(iii) we get $\bD(f),\,\bD(f)^*\in\kL(H^{s-1}(\R))^2$.
The remaining claims follow  from Lemma~\ref{L:MP0}, \eqref{DEest}, and Theorem~\ref{T:L2spec}, by arguing as in the proof of \cite[Theorem~4.2]{MP2022} and Lemma~\ref{L:3}.
\end{proof}

 \section{ Equivalent formulation and proof of the main results}\label{Sec:4}

 In this section we formulate the quasistationary Stokes flow \eqref{STOKES} as an evolution problem for $f$. 
 The main step is established in  Corollary~\ref{C:1},  which provides the unique solvability of  Eq.~\eqref{invertcom},  as announced in Theorem~\ref{T:1}.
 Using this,  in Section~\ref{Sec:4.2} we derive the evolution problem \eqref{NNEP1} for the Stokes flow~\eqref{STOKES}. 
 This is in analogy to the problem \eqref{NNEP2} obtained in \cite{MP2022} for the corresponding 
 two-phase Stokes flow \eqref{2STOKES}. 
 In Section~\ref{Sec:4.3}, Problem~\eqref{NNEP1} is then shown to 
 be the  limit $\mu_+\to0$ of \eqref{NNEP2}. 
This is based on a commutator type identity provided in Proposition~\ref{P:2}. 
Finally, in  Section~\ref{Sec:4.4}, we introduce the general evolution problem \eqref{NNEP} with parameter $\mu_+\geq 0$. This formulation enables us to treat both one- and two-phase flows simultaneously
 for initial data in~$H^s(\R)$, ${s\in(3/2,2)}$, and to establish the main results.

 \subsection{A relation connecting $(\bD(f)[\beta])'$ and $\bD(f)^*[\beta']$}\label{Sec:4.1}\ 
The following identity is,  besides Lemma~\ref{L:3}, the main ingredient in the proof of Corollary~\ref{C:1}.  
\begin{lemma}\label{L:A1}
 Given $f\in H^{\tau}(\R)$, $\tau\in(3/2,2)$, and $\beta\in H^1(\R)^2$, we have~${\bD(f)[\beta]\in H^1(\R)^2}$ with 
 \begin{equation}\label{comder}
 (\bD(f)[\beta])'=-\bD(f)^*[\beta'].
 \end{equation}
 \end{lemma}
In order to prepare the proof of Lemma~\ref{L:A1}, which is presented below, we set
\[
D:=\{(\xi,\xi)\in\R^2\,:\,\xi\in\R\}
\]
and define 
\begin{equation}\label{kernel}
K(\xi,s):=\frac{r_1f'(s)-r_2}{|r|^4}
\left(\begin{array}{cc}
r_1^2&r_1r_2\\
r_1r_2&r_2^2
\end{array}\right),\qquad (\xi,s)\in\R^2\setminus D,
\end{equation}
where $r=r(\xi,s)$ is defined  in \eqref{rxs}.
The double layer potential $\bD(f)$ and its $L_2$-adjoint~$\bD(f)^*$ can now  be expressed as follows:
\begin{align*}
\bD(f)[\beta](\xi)&=\int_\R K(\xi,s)\beta(s)\,ds,\\
\bD(f)^\ast[\beta](\xi)&=\int_\R K(s,\xi)^\top\beta(s)\,ds=\int_\R K(s,\xi)\beta(s)\,ds
\end{align*}
for $\beta\in L_2(\R)^2$, see \eqref{defD}.
Both integrals converge when $f\in H^\tau(\R)$, with $\tau\in(3/2,2),$ since  there exists a constant $C>0$ such that 
\begin{equation}\label{eq:BBG}
\|K(\xi,s)\|_F\leq C |\xi-s|^{\tau-3/2} \qquad\text{for all $(\xi,\, s)\in\R^2\setminus D$.}
\end{equation}
Here, $\|\cdot\|_F$ is  again the Frobenius norm.
Motivated by \eqref{eq:BBG}, we establish the following auxiliary result.

\begin{lemma}\label{exchange}
Let  $A\in {\rm C}(\R^2)\cap {\rm C}^1(\R^2\setminus D)$, $u\in {\rm C}_c(\R)$, and assume  there exist constants~${C>0}$ and $\alpha\in(0,1)$ such that
\[|\partial_\xi A(\xi,s)|\leq C|\xi-s|^{-\alpha} \quad\text{for all $(\xi,s)\in\R^2\setminus D$.}\]
Then, the function $\psi:\R\longrightarrow\R$  given by
\[\psi(\xi):=\int_{\R}A(\xi,s)u(s)\,ds,\qquad \xi\in\R,\]
belongs to ${\rm C}^1(\R)$ and 
\[\psi'(\xi)=\int_\R\p_\xi A(\xi,s)u(s)\,ds,\qquad \xi\in\R.\]
\end{lemma}
\begin{proof}
Given~${\e\in(0,1)}$, let $\psi_\e\in {\rm C}(\R)$ be given by 
\[\psi_\e(\xi):=\int_{\{|\xi-s|>\e\}} A(\xi,s)u(s)\,ds.\]
Since $u$ has compact support and $A\in {\rm C}(\R^2)\cap {\rm C}^1(\R^2\setminus D)$, we have $\psi_\e\in {\rm C}^1(\R)$  and $\psi_\e\to \psi$ for $ \e\to0$ uniformly on compact subsets of $\R$. 
By closedness of the differentiation, the lemma is proved once we show
\[\psi_\e'\to \int_\R\p_\xi A(\cdot,s)u(s)\,ds:=\varphi\quad\text{for $\e\to 0$,}\]
uniformly on compact subsets of $\R$. 
Indeed, given $\xi\in\R$, it holds that
\[\psi_\e'(\xi)=\int_{\{|\xi-s|>\e\}} \p_\xi A(\xi,s)u(s)\,ds+A(\xi,\xi-\e)u(\xi-\e)-A(\xi,\xi+\e)u(\xi+\e),\]
and therefore
\[|(\psi_\e'-\varphi)(\xi)|\leq C\int_{\{|\xi-s|<\e\}}|\xi-s|^{-\alpha}\,ds+|A(\xi,\xi-\e)u(\xi-\e)-A(\xi,\xi+\e)u(\xi+\e)|,\]
which implies the announced convergence. 
\end{proof}

We are now in a position to prove Lemma~\ref{L:A1}.
\begin{proof}[Proof of Lemma~\ref{L:A1}]
 We first establish the result for $\beta\in {\rm C}_c^\infty(\R)^2$.
To this end we define the function $S:=S(f)\in {\rm C}(\R^2)$ by
\[S(f)(\xi,s):=\left\{
\begin{array}{ccl}
\displaystyle\frac{f(\xi)-f(s)}{\xi-s},&&\text{if $\xi\neq s$,}\\[2ex]
f'(\xi),&&\text{if $\xi=s$.}
\end{array}\right.\]
The function   $S$ is continuously differentiable in $\R^2\setminus D$, where again~${D:=\{(\xi,\xi)\,:\,\xi\in\R\}}$,  with partial derivatives expressed, in the notation~\eqref{rxs}, as
\begin{equation*}
\p_\xi S(\xi,s)=\frac{f'(\xi)r_1-r_2}{r_1^2}\qquad\text{and}\qquad\p_sS(\xi,s)=\frac{-f'(s)r_1+r_2}{r_1^2}.
\end{equation*}
Now (with arguments $(\xi,s)$ partly suppressed), the kernel $K$ defined in \eqref{kernel} can be expressed as
\begin{align*}
K(\xi,s)&=\frac{1}{(1+S^2)^2}
\left(\begin{array}{cc}
1&S\\
S&S^2
\end{array}\right)\,\frac{r_1f'(s)-r_2}{r_1^2}
=-\frac{1}{(1+S^2)^2}
\left(\begin{array}{cc}
1&S\\
S&S^2
\end{array}\right)\p_s S(\xi,s)\\[1ex]
&=-\p_s G\circ S(\xi,s),
\end{align*}
where $G\in {\rm C}^\infty(\R,\R^{2\times2})$ is a primitive of the smooth matrix valued function 
\[\Big[r\longmapsto\frac{1}{(1+r^2)^2}
\left(\begin{array}{cc}
1&r\\
r&r^2
\end{array}
\right)
\Big]:\R\longrightarrow\R^{2\times 2}.\]
Thus, after integration by parts, 
\[\bD(f)[\beta](\xi)=\int_\R G(S(\xi,s))\beta'(s)\,ds.\]
We next observe that $A:=G\circ S$ is continuous on $\R^2$ and  $A\in {\rm C}^1(\R^2\setminus D)$ with
\begin{align*}
\p_\xi A(\xi,s)&=\frac{1}{(1+S^2)^2}
\left(\begin{array}{cc}
1&S\\
S&S^2
\end{array}\right)\p_\xi S(\xi,s)=-
\frac{1}{(1+S^2)^2}
\left(\begin{array}{cc}
1&S\\
S&S^2
\end{array}\right)\,\frac{-r_1f'(\xi)+r_2}{r_1^2}\\[1ex]
&=-K(s,\xi)=-K(s,\xi)^\top.
\end{align*}
Moreover, since  $f\in {\rm BUC}^{\tau-1/2}(\R)$, setting $\alpha:=\tau-3/2\in(0,1)$, we  obtain from~\eqref{eq:BBG} that there exists $C>0$ such that
\begin{equation*}
\|\p_\xi A(\xi,s)\|_F\leq C |\xi-s|^{-\alpha} \qquad\text{for all $(\xi,\, s)\in\R^2\setminus D$.}
\end{equation*}
We may now infer from  Lemma \ref{exchange}  that  $\bD(f)[\beta]\in{\rm C^1}(\R)$ and
\[(\bD(f)[\beta])'(\xi)=\int_\R\p_\xi A(\xi,s) \beta'(s)\,ds=-\int_\R K(s,\xi)^\top\beta(s)\,ds=-\bD(f)^\ast[\beta'](\xi),\quad \xi\in\R.\]

Let now $\beta\in H^1(\R)^2$ arbitrary, and let $(\beta_n)$ be a sequence in ${\rm C}^\infty_c(\R)^2$ with 
$\beta_n\to\beta$ in~${H^1(\R)^2}$. 
Then, by $L_2$-continuity of $\bD(f)$, see~\eqref{CDD*},
\[\bD(f)[\beta_n]\to\bD(f)[\beta]\quad\text{in $L_2(\R)^2$}\]
and, by $L_2$-continuity of $\bD(f)^\ast$, see~\eqref{CDD*},
\[(\bD(f)[\beta_n])'= -\bD(f)^\ast[\beta'_n]\to-\bD(f)^\ast[\beta']\quad\text{in $L_2(\R)^2$}.\]
Now the result follows by closedness of the derivative operator.
\end{proof}

We are now in a position to establish the solvability of \eqref{invertcom}.
 \begin{cor}\label{C:1}
 Given $f\in H^3(\R)$ and $g\in H^2(\R)^2$, let $\beta\in H^2(\R)^2$ denote the unique solution to the equation
 \begin{equation}\label{Dual1}
\Big(\frac{1}{2}+ \bD(f)\Big)[\beta]= g.
\end{equation}
Then $\alpha:=\beta'\in H^1(\R)^2$ is the unique solution to
 \begin{equation}\label{Dual2}
\Big(\frac{1}{2}- \bD(f)^*\Big)[\alpha]=g'.
\end{equation} 
 \end{cor}
 \begin{proof}
 The claim is a direct consequence of Lemma~\ref{L:3} and Lemma~\ref{L:A1}.
 \end{proof}

  \subsection{The evolution problem for $f$}\label{Sec:4.2}
  Let $ T_+>0$ and $(f,v,p)$ be a solution to \eqref{STOKES} such that for all~$t\in(0,T_+)$ we have $f(t)\in H^3(\R)$ and
 \[
v(t)\in {\rm C}^2(\0(t))\cap {\rm C}^1 (\ov{\0(t)}),\quad  p(t)\in {\rm C}^1(\0(t))\cap {\rm C} (\ov{\0(t)}),\quad   v(t)|_{\G(t)} \circ\Xi_{f(t)}\in H^2(\R)^2.
 \] 
  Define $\beta(t):=\beta(f(t))\in H^{2}(\R)^2$ by
\begin{equation}\label{Phi2}
 \beta(t):= \Big(\frac{1}{2}+ \bD(f(t))\Big)^{-1}[g(t)],
\end{equation}
with $g(t)=g(f(t))$  as defined in \eqref{defgn}.
 Then, by Corollary~\ref{C:1}, $\beta(t)'$ is the unique solution to
\[\Big(\frac{1}{2}- \bD(f(t))^\ast\Big)[\beta(t)']=g(t)',\]
with the prime denoting the spatial derivative along $\R$. 
  Theorem~\ref{T:1}, in particular  \eqref{vgamma},  then implies that
 \begin{equation}\label{veltra}
 v(t)|_{\G(t)} \circ\Xi_{f(t)}=\frac{\sigma}{\mu}\bV(f(t))[\beta(t)]
 \end{equation}
 where, given $f\in H^3(\R)$, the operator   $\bV(f)\in\kL(H^{2}(\R)^2)$  (see \eqref{Bnmk})  is defined by
\begin{equation}\label{trv}
\begin{aligned}
\mathbb{V}(f)[\beta]&:=\frac{1}{4}\begin{pmatrix}
B_{2,2}^0(f)-B_{0,2}^0(f) &B_{3,2}^0(f)-B_{1,2}^0(f)\\[1ex]
B_{3,2}^0(f)-B_{1,2}^0(f)&-3B_{2,2}^0(f)-B_{0,2}^0(f)
\end{pmatrix}\begin{pmatrix}
\beta_1\\
\beta_2
\end{pmatrix}\\[1ex]
&\qquad+\frac{1}{4}\begin{pmatrix}
-B_{3,2}^0(f)-3B_{1,2}^0(f)&-B_{2,2}^0(f)+B_{0,2}^0(f) &\\[1ex]
-B_{2,2}^0(f)+B_{0,2}^0(f)&-B_{3,2}^0(f)+B_{1,2}^0(f)
\end{pmatrix}\begin{pmatrix}
f'\beta_1\\
f'\beta_2
\end{pmatrix}
\end{aligned} 
\end{equation}
for $\beta=(\beta_1,\beta_2)^\top\in H^{2}(\R)^2 $.
Recalling \eqref{StP}$_5$ and \eqref{IC}, we may thus recast~\eqref{STOKES} as the following evolution problem
 \begin{equation}\label{NNEP1}
 \frac{df}{dt}=\frac{\sigma}{\mu}\mathbb{V}(f)[\beta]\cdot (-f',1),\quad t>0,\qquad f(0)=f^{(0)},
 \end{equation}
 where the evolution equation should be  satisfied pointwise with values in  $H^2(\R)$.

\subsection{Problem \eqref{NNEP1} as the limit $\mu_+\to 0$ of the two-phase Stokes problem}\label{Sec:4.3}
In \cite{MP2022} it is shown that if $T_{+}>0$ and $(f, w^\pm,q^\pm)$ is a solution to the two-phase quasistationary Stokes flow~\eqref{2STOKES} such that for all
$t\in(0,T_{+})$ we have $f(t)\in H^3(\R)$, $ w^\pm(t)|_{\G(t)}\circ\Xi_{f(t)}\in H^2(\R)^2$, and 
\begin{align*}
 w^\pm(t)\in {\rm C}^2(\Omega^\pm(t))\cap {\rm C}^1(\overline{\Omega^\pm(t)}),\quad q^\pm(t)\in {\rm C}^1(\Omega^\pm(t))\cap {\rm C}(\overline{\Omega^\pm(t)}),
\end{align*}
then $f=f(t)$ solves the evolution problem
\begin{equation}\label{NNEP2}
 \frac{df}{dt}=\frac{\sigma}{\mu^++\mu^-}\gamma\cdot (-f',1),\quad t>0,\qquad f(0)=f^{(0)},
 \end{equation}
where $\gamma(t):=\gamma(f(t))\in H^2(\R)^2$ is given by
\begin{align}\label{TOSO}
\gamma(t):=\Big(\frac{1}{2}+a_\mu\bD(f(t))\Big)^{-1}[\bV(f(t))[g(t)]]
\end{align} 
 and
\[a_\mu:=\frac{\mu^+-\mu^-}{\mu^++\mu^-}\in(-1,1).\]
In \eqref{TOSO}, $g(f(t))\in H^2(\R)^2$  and $\bV(f(t))\in\kL(H^{2}(\R)^2)$, $t\in(0,T_+)$, are   defined in  \eqref{defgn} and~\eqref{trv},
 respectively, and the evolution equation should be again satisfied pointwise  with values in~$H^2(\R), $ see~\cite[Theorem 4.5]{MP2022}.

 We next prove  that  the formulation~\eqref{NNEP1} coincides with the limit $\mu_+\to0$  of~\eqref{NNEP2}  with~${\mu^-=\mu}$ fixed.
 Then $a_\mu\to -1$, and as Lemma~\ref{L:3} and~\cite[Theorem 4.5]{MP2022} show, this limit can be taken by continuity in \eqref{NNEP2}.

 \begin{prop}\label{P:2}
  Given $f\in H^3(\R)$, it holds that 
 \begin{equation}\label{eq:co}
 \mathbb{V}(f)\Big(\frac{1}{2}+ \bD(f)\Big)^{-1} =\Big(\frac{1}{2}- \bD(f)\Big)^{-1}\mathbb{V}(f). 
 \end{equation}
 \end{prop}
 \begin{proof}
 We are going to prove the more compact identity
 \begin{equation}\label{eq:co2}
\bV(f)\bD(f)+\bD(f) \bV(f)=0,
 \end{equation}
 which is equivalent to \eqref{eq:co} in view of Lemma~\ref{L:3}.

Let $\beta\in H^2(\R)^2$ be arbitrary and let $(v,p):=(v,p)(f)[\beta]$ be defined by \eqref{Solutions}  (with $\mu=1$). 
Then, as shown in Theorem~\ref{T:1}, it holds that 
 \[
v\in {\rm C}^2(\0)\cap {\rm C}^1 (\ov\0),\quad  p\in {\rm C}^1(\0)\cap {\rm C} (\ov\0),\quad   v|_\G\circ\Xi=\bV(f)[\beta] \in H^2(\R)^2,
 \] 
and, recalling  also Corollary~\ref{C:1}, $(v,p)$ solves the boundary value problem
 \begin{equation}\label{SBVP21}
\left.
\begin{array}{rclll}
\Delta v-\nabla p&=&0&\mbox{in $\Omega$,}\\
\vdiv v&=&0&\mbox{in $\Omega$,}\\{}
T_1(v, p)\tilde\nu&=&(\gamma'/\omega)\circ\Xi^{-1}&\mbox{on $\Gamma$,}\\
(v, p)(x)&\to&0&\mbox{for $|x|\to\infty$,}
\end{array}\right\}
\end{equation}
where
\begin{equation}\label{ran0}
\gamma:= \Big(\frac{1}{2}+\bD(f)\Big)[\beta]\in H^2(\R)^2.
\end{equation}
We next define~${(w^\pm,q^\pm):\0^\pm\longrightarrow\R^2\times\R}$ by $(w^-,w^+)=(  v,0)$ and $(q^-,q^+)=(p,0)$.
Then
\begin{equation}\label{regwq}
w^\pm\in {\rm C}^2(\0^\pm)\cap {\rm C}^1 (\ov{\0^\pm}),\quad  q^\pm\in {\rm C}^1(\0^\pm)\cap {\rm C} (\ov{\0^\pm}),\quad   w^\pm|_\G \circ\Xi\in H^2(\R)^2,
 \end{equation}
and  $(w^\pm,q^\pm)$ solves the boundary value problem
\be\label{bvpaux21}
\left.\begin{array}{rcll}
\Delta w^\pm-\nabla q^\pm&=&0&\mbox{in $\Omega^\pm$,}\\
\vdiv w^\pm&=&0&\mbox{in $\Omega^\pm$,}\\{}
[w]&=&-v|_\Gamma&\mbox{on $\Gamma$,}\\
{}[T_1(w,q)]\tilde \nu &=&-(\gamma'/\omega)\circ\Xi^{-1}&\mbox{on $\Gamma$,}\\
(w^\pm,q^\pm)(x)&\to&0&\mbox{for $|x|\to\infty$.}
\end{array}\right\}
\ee
Similarly  as in the proof of \cite[Proposition 5.1]{MP2022}, we decompose $(w^\pm,q^\pm)$ as a sum 
\[
(w^\pm,q^\pm)=(w^\pm_s,q^\pm_s)+(w^\pm_d,q^\pm_d),
\]
where $(w^\pm_s,q^\pm_s)$ solves  
\be\label{bvpaux21a}
\left.\begin{array}{rcll}
\Delta w_s^\pm-\nabla q_s^\pm&=&0&\mbox{in $\Omega^\pm$,}\\
\vdiv w_s^\pm&=&0&\mbox{in $\Omega^\pm$,}\\{}
[w_s]&=&0&\mbox{on $\Gamma$,}\\
{}[T_1(w_s,q_s)]\tilde \nu&=&-(\gamma'/\omega)\circ\Xi^{-1}&\mbox{on $\Gamma$,}\\
(w_s^\pm,q_s^\pm)(x)&\to&0&\mbox{for $|x|\to\infty$.}
\end{array}\right\}
\ee
The system \eqref{bvpaux21a} has been studied in \cite{MP2021} and, according to \cite[Theorem~2.1 and Lemma~A.1]{MP2021}, it has a unique solution which satisfies
\begin{equation}\label{ran1}
w_s^\pm|_\Gamma\circ\Xi=\bV(f)[\gamma]
\end{equation}
  and which has the same regularity as $(w,q)$, see \eqref{regwq}. 
Consequently, $(w^\pm_d,q^\pm_d)$ enjoys also the regularity \eqref{regwq} and moreover it solves 
\be\label{bvpaux22}
\left.\begin{array}{rcll}
\Delta w^\pm_d-\nabla q^\pm_d&=&0&\mbox{in $\Omega^\pm$,}\\
\vdiv w^\pm_d&=&0&\mbox{in $\Omega^\pm$,}\\{}
[w_d]&=&-v|_\Gamma&\mbox{on $\Gamma$,}\\
{}[T_1(w_d,q_d)]\tilde \nu&=&0&\mbox{on $\Gamma$,}\\
(w_d^\pm,q_d^\pm)(x)&\to&0&\mbox{for $|x|\to\infty$.}
\end{array}\right\}
\ee
Since $v|_\Gamma\circ\Xi \in H^2(\R)^2$, also \eqref{bvpaux22} has a unique solution, see \cite[Proposition 2.1]{MP2022}, and, according to \cite[Lemma A.1]{MP2022} we have
\begin{equation}\label{ran2}
w^-_d|_\Gamma\circ \Xi=\Big(\frac{1}{2}+\bD(f)\Big)[v|_\Gamma\circ\Xi].
\end{equation}
Since $\bV(f)[\beta]=v|_\Gamma\circ\Xi=(w^-_s+w^-_d)|_\Gamma\circ\Xi,$ we infer from \eqref{ran0}, \eqref{ran1}, and \eqref{ran2} that 
\[
\bV(f)[\beta]=\bV(f)\Big(\frac{1}{2}+\bD(f)\Big)[\beta]+\Big(\frac{1}{2}+\bD(f)\Big)[\bV(f)[\beta]] \qquad\text{for all $\beta\in H^2(\R)^2$,}
\]
and \eqref{eq:co2} is now a direct consequence  of this identity.
\end{proof}

\subsection{The common equivalent formulation for \eqref{STOKES} and \eqref{2STOKES}}\label{Sec:4.4}\ 
 We are now able to simultaneously consider the evolution equations \eqref{NNEP1} and \eqref{NNEP2} that respectively encode the one-phase problem \eqref{STOKES} and the two-phase problem \eqref{2STOKES}. 
 For this purpose we set 
\[
\mu^-=\mu  
\]
and view $\mu^+\in[0,\infty)$ as a parameter.

For $f\in H^3(\R)$, we set 
\begin{equation}\label{Phi}
 \Phi(\mu^+,f):=\frac{\sigma}{\mu^++\mu}\Big(\frac{1}{2}+\frac{\mu^+-\mu}{\mu^++\mu}\bD(f)\Big)^{-1}[\bV(f)[g(f)]]\cdot (-f',1),
 \end{equation}
with  $g(f)$  defined in \eqref{defgn} and $\bV(f)$ in \eqref{trv}.

Now, in view of Proposition~\ref{P:2}, the parameter dependent evolution equation
\begin{equation}\label{NNEP}
 \frac{df}{dt}=\Phi(\mu^+,f),\quad t\geq 0,\qquad f(0)=f^{(0)},
 \end{equation}
 is identical to \eqref{NNEP1} for $\mu^+=0$ and to \eqref{NNEP2} for $\mu^+>0$.

Though \eqref{NNEP} has been derived under the assumption that the solution lies in $ H^3(\R)$ for positive times,  this equation may be viewed in a more general analytic setting.
Indeed, given~$f\in H^{s}(\R)$, $s\in(3/2,2)$, we first deduce from \cite[Lemma~3.5]{MP2021} that~${g(f)\in H^{s-1}(\R)^2}$ and
\begin{equation}\label{regg}
g\in {\rm C}^\infty(H^{s}(\R),H^{s-1}(\R)^2).
\end{equation} 
Moreover, since   
\[
[f\longmapsto B^0_{n,m}(f)]\in{\rm C}^\infty(H^s(\R),\kL( H^{s-1}(\R))),\qquad n,\, m\in\N,
\] 
  see \cite[Corollary C.5]{MP2021}, 
 \eqref{DFB}$_1$  and \eqref{trv}  yield
  \begin{equation}\label{regBV}
[f\longmapsto \bD(f)],\, [f\longmapsto \bV(f)] \in{\rm C}^\infty(H^s(\R),\kL( H^{s-1}(\R)^2)).
\end{equation}
The properties \eqref{regg}, \eqref{regBV}, Lemma~\ref{L:4} (for $\mu^+=0$) and \cite[Theorem 4.2]{MP2022} (for $\mu^+>0$) ensure now that the operator~${\Phi:[0,\infty)\times H^s(\R)\longrightarrow H^{s-1}(\R)}$ is well-defined.
Moreover, the smoothness of the function which maps an isomorphism to its inverse together with \eqref{regg}  and \eqref{regBV} implies that 
  \begin{equation}\label{regPhi}
  \Phi\in {\rm C}^\infty([0,\infty)\times H^s(\R), H^{s-1}(\R)).
  \end{equation}
  We conclude this section with the observation that $\Phi(\mu^+,\cdot)$ maps bounded sets of $H^s(\R)$ to bounded sets of $H^{s-1}(\R).$
  This property is a consequence of the fact that $g$ maps  bounded sets of~$H^{s}(\R)$ to bounded sets of $H^{s-1}(\R)^2$, of Lemma~\ref{L:MP0}~(iii), 
  and of Lemma~\ref{L:4} (for $\mu^+=0$) or \cite[Theorem 4.2]{MP2022} (for $\mu^+>0$).

 \subsection{The proofs of the main results}\label{Sec:4.5}
 In the case  $\mu^+>0$, we  shown in \cite[Theorem~6.1]{MP2022} 
 that the Fr\'echet derivative $\p_f\Phi(\mu^+,f)$ is, for each $f\in H^{s}(\R),$ the generator of an analytic semigroup in $\kL(H^{s-1}(\R)).$
 This property, the observation that $\Phi(\mu^+,\cdot)$ maps bounded sets of~${H^s(\R)}$ to bounded sets of $H^{s-1}(\R),$ 
 the smoothness of $\Phi(\mu^+,\cdot)$, the fully nonlinear parabolic theory from \cite{L95}, and a parameter
  trick used also for other problems, see e.g. \cite{An90, ES96, PSS15, MBV19},
  were then exploited in \cite{MP2022} to establish the well-posedness of the two-phase quasistationary Stokes flow \eqref{2STOKES}, see \cite[Theorem~1.1]{MP2022}.
 All these properties are satisfied also when $\mu^+=0$.
 Indeed, using Lemma~\ref{L:4} instead of \cite[Theorem~4.2]{MP2022}, the arguments in the proof of 
 \cite[Theorem~6.1]{MP2022} remain valid also if $\mu^+=0$, hence also $\p_f\Phi(0,f)$ is the generator of an analytic semigroup in~${\kL(H^{s-1}(\R))}$ for each $f\in H^{s}(\R)$.
  In view of these facts we have: 
  
  \begin{proof}[Proof of Theorem~\ref{MT1}]
  The proof is identical to that of  \cite[Theorem~1.1]{MP2022} and therefore we omit the details.
  \end{proof}
  
  For the  proof of Theorem~\ref{MT2} we consider $\mu^+$ as a parameter and we use a result on the continuous dependence of the solutions to abstract parabolic problems on  parameters 
  provided in \cite[Theorem 8.3.2]{L95}
    \begin{proof}[Proof of Theorem~\ref{MT2}]
 In view of \eqref{regPhi} and of the fact that Fr\'echet derivative $\p_f\Phi(\mu^+,f)$ is, for each $(\mu^+,f)\in[0,\infty)\times H^{s}(\R),$ the generator of an analytic semigroup in $\kL(H^{s-1}(\R))$ we find that all the assumptions of
  \cite[Theorem 8.3.2]{L95} are satisfied in the context of \eqref{NNEP}.

Let thus~$f(\cdot;f^{(0)}):[0,T_+( f^{(0)}))\longrightarrow\R$ denote the maximal solution to \eqref{NNEP} with $\mu^+=0$ and fix~${T\in(0,T_+( f^{(0)}))}$.
In view of \cite[Theorem 8.3.2]{L95} there exist constants $\varepsilon>0$ and~${M>0}$ such that for each $\mu^+ \in(0,\varepsilon]$ 
  the solution $f_{\mu^+}(\cdot; f^{(0)})$ to~\eqref{NNEP} found in \cite[Theorem 6.1]{MP2022} satisfies $T_{+,\mu^+}( f^{(0)})>T$ and
  \[
\big\|f(\cdot ;f^{(0)})- f_{\mu^+}(\cdot; f^{(0))}\big\|_{{\rm C}([0,T], H^s(\R))}
+\Big\|\frac{d}{dt}\big(f(\cdot; f^{(0)})- f_{\mu^+}(\cdot; f^{(0)})\big)\Big\|_{{\rm C}([0,T], H^{s-1}(\R))}\leq M\mu^+.
  \]
  This completes the proof.
  \end{proof}

 \bibliographystyle{siam}
 \bibliography{MP3}

\begin{thebibliography}{10}

\bibitem{AM22}
{\sc H.~Abels and B.-V. Matioc}, {\em {Well-posedness of the Muskat problem in
  subcritical $L_p$-Sobolev spaces}}, European J. Appl. Math., 33 (2022),
  pp.~224--266.

\bibitem{Am95}
{\sc H.~Amann}, {\em {Linear and Quasilinear Parabolic Problems. {V}ol. {I}}},
  vol.~89 of {Monographs in Mathematics}, Birkh\"auser Boston, Inc., Boston,
  MA, 1995.
\newblock Abstract linear theory.

\bibitem{An90}
{\sc S.~B. Angenent}, {\em {Nonlinear analytic semiflows}}, Proc. Roy. Soc.
  Edinburgh Sect. A, 115 (1990), pp.~91--107.

\bibitem{BaDu98}
{\sc A.~Badea and J.~Duchon}, {\em Capillary driven evolution of an interface
  between viscous fluids}, Nonlinear Anal., 31 (1998), pp.~385--403.

\bibitem{BM22}
{\sc J.~Bierler and B.-V. Matioc}, {\em {The multiphase Muskat problem with
  equal viscosities in two dimensions}}, Interfaces Free Bound., 24 (2022),
  pp.~163--196.

\bibitem{ES96}
{\sc J.~Escher and G.~Simonett}, {\em {Analyticity of the interface in a free
  boundary problem}}, Math. Ann., 305 (1996), pp.~439--459.

\bibitem{FKV88}
{\sc E.~B. Fabes, C.~E. Kenig, and G.~C. Verchota}, {\em The {D}irichlet
  problem for the {S}tokes system on {L}ipschitz domains}, Duke Math. J., 57
  (1988), pp.~769--793.

\bibitem{FR02b}
{\sc A.~Friedman and F.~Reitich}, {\em {Quasi-static motion of a capillary
  drop. {II}. {T}he three-dimensional case}}, J. Differential Equations, 186
  (2002), pp.~509--557.

\bibitem{Fr02a}
\leavevmode\vrule height 2pt depth -1.6pt width 23pt, {\em {Quasistatic motion
  of a capillary drop. {I}. {T}he two-dimensional case}}, J. Differential
  Equations, 178 (2002), pp.~212--263.

\bibitem{G17}
{\sc F.~Gancedo}, {\em A survey for the {M}uskat problem and a new estimate},
  SeMA J., 74 (2017), pp.~21--35.

\bibitem{GL20}
{\sc R.~Granero-Belinch\'{o}n and O.~Lazar}, {\em Growth in the {M}uskat
  problem}, Math. Model. Nat. Phenom., 15 (2020), pp.~Paper No. 7, 23.

\bibitem{PG97}
{\sc M.~G\"{u}nther and G.~Prokert}, {\em Existence results for the
  quasistationary motion of a free capillary liquid drop}, Z. Anal.
  Anwendungen, 16 (1997), pp.~311--348.

\bibitem{Lad63}
{\sc O.~A. Ladyzhenskaya}, {\em The mathematical theory of viscous
  incompressible flow}, Revised English edition. Translated from the Russian by
  Richard A. Silverman, Gordon and Breach Science Publishers, New York-London,
  1963.

\bibitem{L95}
{\sc A.~Lunardi}, {\em {Analytic Semigroups and Optimal Regularity in Parabolic
  Problems}}, {Progress in Nonlinear Differential Equations and their
  Applications, 16}, Birkhäuser Verlag, Basel, 1995.

\bibitem{MBV18}
{\sc B.-V. Matioc}, {\em {Viscous displacement in porous media: the Muskat
  problem in 2D}}, Trans. Amer. Math. Soc., 370 (2018), pp.~7511--7556.

\bibitem{MBV19}
\leavevmode\vrule height 2pt depth -1.6pt width 23pt, {\em {The Muskat problem
  in two dimensions: equivalence of formulations, well-posedness, and
  regularity results}}, Anal. PDE, 12 (2019), pp.~281--332.

\bibitem{MP2021}
{\sc B.-V. Matioc and G.~Prokert}, {\em Two-phase {S}tokes flow by capillarity
  in full 2d space: an approach via hydrodynamic potentials}, Proc. Roy. Soc.
  Edinburgh Sect. A, 151 (2021), pp.~1815--1845.

\bibitem{MP2022}
\leavevmode\vrule height 2pt depth -1.6pt width 23pt, {\em {Two-phase Stokes
  flow by capillarity in the plane: The case of different viscosities}}, NoDEA
  Nonlinear Differential Equations Appl., 29 (2022), pp.~Paper No. 54, 34.

\bibitem{PSS15}
{\sc J.~Pr\"uss, Y.~Shao, and G.~Simonett}, {\em {On the regularity of the
  interface of a thermodynamically consistent two-phase {S}tefan problem with
  surface tension}}, Interfaces Free Bound., 17 (2015), pp.~555--600.

\bibitem{PS16}
{\sc J.~Pr\"uss and G.~Simonett}, {\em Moving Interfaces and Quasilinear
  Parabolic Evolution Equations}, vol.~105 of Monographs in Mathematics,
  Birkh\"auser/Springer, [Cham], 2016.

\bibitem{S99}
{\sc V.~A. Solonnikov}, {\em {On quasistationary approximation in the problem
  of motion of a capillary drop}}, in {Topics in Nonlinear Analysis}, vol.~35
  of {Progr. Nonlinear Differential Equations Appl.}, Birkhäuser, Basel, 1999,
  pp.~643--671.

\bibitem{Solo99}
{\sc V.~A. Solonnikov}, {\em On the justification of the quasistationary
  approximation in the problem of motion of a viscous capillary drop},
  Interfaces Free Bound., 1 (1999), pp.~125--173.

\end{thebibliography}
\end{document}